\documentclass[12pt,reqno]{amsart}
\allowdisplaybreaks

\usepackage[english]{babel}
\usepackage[utf8]{inputenc}
\usepackage[T1]{fontenc}
\usepackage{amssymb,amsmath,amstext,amsfonts,amsthm}
\usepackage{mathtools}
\usepackage{mathrsfs}
\usepackage{verbatim}
\usepackage{relsize}
\usepackage{graphicx,color}
\usepackage[dvipsnames]{xcolor}
\usepackage{mathabx}
\usepackage[all]{xy}
\usepackage[colorlinks=true, urlcolor=blue, linkcolor=blue, citecolor=blue]{hyperref}
\usepackage{subcaption}
\usepackage{thm-restate}
\usepackage{algorithm}
\usepackage{algpseudocode}
\usepackage{setspace}

\usepackage{nicematrix}
\usepackage{blkarray}

\usepackage{fancyvrb}
\usepackage{tikz-cd}
\usepackage{comment}
\usepackage{hhline}
\usepackage{diagbox}
\usepackage{epigraph}
\usepackage{listings}
\usepackage{cleveref}
\crefname{thm}{theorem}{theorems}
\Crefname{thm}{Theorem}{Theorems}
\crefname{claim}{claim}{claims}
\Crefname{claim}{Claim}{Claims}
\definecolor{codedarkgreen}{RGB}{51, 133, 4}
\definecolor{codemaroon}{RGB}{133, 5, 63}
\definecolor{codeteal}{RGB}{0, 128, 96}
\lstdefinelanguage{Macaulay2}{
basicstyle=\small\ttfamily,
alsoletter=",
classoffset=1,
keywords={matrix,minors,gb,transpose,det,ideal,apply,subsets,ker,gens,fold,flatten,entries},
keywordstyle={\color{blue}},
classoffset=2,
morekeywords={from, to, list},
keywordstyle={\color{codemaroon}},
classoffset=3,
morekeywords={QQ},
keywordstyle={\color{codedarkgreen}},
classoffset=4,
morekeywords={MonomialOrder},
keywordstyle={\color{codeteal}},
xleftmargin=1.5cm,
xrightmargin=1em,
columns=fullflexible,
keepspaces=true,
stepnumber=1,
numbers=none,
captionpos=b,
showspaces=false,
frame=none
}
\usepackage[left=25mm, right=25mm, top=30mm, bottom=30mm]{geometry}

\DeclareMathOperator{\DL}{DL}

\DeclareMathOperator{\Det}{Det}

\DeclareMathOperator{\rank}{rank}

\numberwithin{equation}{section}

\theoremstyle{plain}
\newtheorem{Theorem}{Theorem}
\numberwithin{Theorem}{section}
\newtheorem{Corollary}[Theorem]{Corollary}
\newtheorem{Lemma}[Theorem]{Lemma}
\newtheorem{Proposition}[Theorem]{Proposition}
\newtheorem{Conjecture}[Theorem]{Conjecture}
\theoremstyle{definition}
\newtheorem{Definition}[Theorem]{Definition}
\newtheorem{Example}[Theorem]{Example}

\theoremstyle{remark}
\newtheorem{Remark}[Theorem]{Remark}

\newcommand{\CC}{\mathbb{C}}
\newcommand{\NN}{\mathbb{N}}

\newcommand{\R}{\mathbb{R}}

\newcommand{\bi}{\textbf{i}}
\newcommand{\bj}{\textbf{j}}

\newcommand{\sG}{\mathcal{G}}
\newcommand{\sI}{\mathcal{I}}
\newcommand{\sJ}{\mathcal{J}}
\newcommand{\sS}{\mathcal{S}}

\newcommand{\sH}{\mathcal{H}}

\newcommand{\sL}{\mathcal{L}}
\newcommand{\sK}{\mathcal{K}}

\newcommand{\W}{\mathrm{W}}

\newcommand{\SO}{\mathrm{SO}}

\newcommand{\T}{\mathcal{T}}

\newcommand\subsetsim{\mathrel{\ooalign{\raise0.2ex\hbox{$\subset$}\cr\hidewidth\raise-0.8ex\hbox{\scalebox{0.9}{$\sim$}}\hidewidth\cr}}}

\makeatletter
\@namedef{subjclassname@2020}{
	\textup{2020} Mathematics Subject Classification}
\makeatother
\title{Decomposing tensors via rank-one approximations} 
\author{\'{A}lvaro Ribot}
\address{\'{A}lvaro Ribot, Harvard University}
\email{aribotbarrado@g.harvard.edu}
\author{Emil Horobet}
\address{Emil Horobet, Sapientia Hungarian University of Transylvania}
\email{horobetemil@ms.sapientia.ro}
\author{Anna Seigal}
\address{Anna Seigal, Harvard University}
\email{aseigal@seas.harvard.edu}
\author{Ettore Teixeira Turatti}
\address{Ettore Teixeira Turatti, UIT - The Arctic University of Norway}
\email{ettore.t.turatti@uit.no}


\begin{document}

\begin{abstract}
Matrices can be decomposed via rank-one approximations:
the best rank-one approximation is a singular vector pair, and the singular value decomposition writes a matrix as a sum of singular vector pairs. 
The singular vector tuples of a tensor are the critical points of its best rank-one approximation problem.
In this paper, we study tensors that can be decomposed via successive rank-one approximations: compute a singular vector tuple, subtract it off, compute a singular vector tuple of the new deflated tensor, and repeat. 
The number of terms in such a decomposition may exceed the tensor rank. 
Moreover, the decomposition may depend on the order in which terms are subtracted. We show that the decomposition is valid independent of order if and only if all singular vectors in the process are orthogonal in at least two factors. We study the variety of such tensors. We lower bound its dimension, showing that it is significantly larger than the variety of odeco tensors.
\end{abstract}

\maketitle

\section{Introduction}
For matrices, critical rank-one approximations combine to give the singular value decomposition: a matrix $M$ can be decomposed via successive critical rank-one approximations as~$M - \lambda_1 u_1 \otimes v_1 - \cdots - \lambda_r u_r \otimes v_r = 0$, where $r = \rank(M)$. Each term $\lambda_i u_i \otimes v_i$ is a critical rank-one approximation of both the original matrix $M$ and the deflated matrix~$M - \sum_{j < i} \lambda_j u_j \otimes v_j$. According to the Schmidt--Eckart--Young theorem \cite{eckart1936approximation}, truncating to the $k$ largest $\lambda_i$'s gives the best rank-$k$ approximation of $M$. Thus, the best rank-$k$ approximation problem reduces to iteratively solving best rank-one approximation problems. 

Low-rank approximation of tensors can also be computed via rank-one updates. Though this approach may not lead to the best low-rank approximation \cite{kolda2001orthogonal, stegeman2010subtracting}, it has proven successful in practice \cite{zhang2001rank, sorber2013optimization, anandkumar2014guaranteed}. Numerically, a rank-one approximation can be computed efficiently with the power method and its variants \cite{de1995higher, kolda2011shifted}, for which there are convergence guarantees \cite{uschmajew2015new}.

 For tensors of order 3 or higher, a best rank-$r$ approximation may not even exist for $r>1$ \cite{de2008tensor}, but a best rank-one approximation always exists because the set of rank-one tensors is closed. However, note that computing a best rank-one approximation is NP-hard \cite{Hillar13}. A global approach to computing a best rank-one approximation of a tensor is to compute the critical points of the distance function to the variety of rank-one tensors and to choose a critical point whose residual is smallest. A generic tensor has a finite number of such critical points \cite{FO14}. These are the singular vector tuples of the tensor~\cite{lim2005singular}. We use the terms critical rank-one approximation and singular vector tuple interchangeably. 

A minimal decomposition via rank-one approximations
is known to exist only on a measure zero set; see the Schmidt--Eckart--Young decomposition of~\cite{vannieuwenhoven2014generic}.
In this article, rather than producing a minimal decomposition via rank-one approximations, we are interested in tensors that are a finite sum of critical rank-one approximations. The decomposition need not be of minimal length. That is, we seek those tensors $\T$ for which
there exists a decomposition $\T = \sum_{i=1}^r x_i$ such that each $x_j$ is a critical rank-one approximation of~$\T - \sum_{i<j} x_i$.
These are the tensors for which
the following algorithm terminates for some rank-one approximations~$x_i$.

\begin{algorithm}[htbp]
	\caption{Decompose $\T$ via rank-one approximations}
    \label{alg:1}
	\begin{algorithmic}[1]
		\renewcommand{\algorithmicrequire}{\textbf{Input:}}
		\Require Tensor $\T \in \R^{n_1} \otimes \cdots \otimes \R^{n_d}$
		\State $i = 1$ 
  \State $\mathcal{S} = \T$ 
\While {$\mathcal{S} \neq 0$}
\State Compute a critical rank-one approximation $x_i = x_i^{(1)} \otimes \cdots \otimes x_i^{(d)}$ of $\mathcal{S}$
\State Deflate $\mathcal{S} \gets \mathcal{S} - x_i$ 
\State $i \gets i + 1$ 
\EndWhile
		 \renewcommand{\algorithmicrequire}{\textbf{Output:}}
		\Require Decomposition $\T = \sum_i  x_i$.
	\end{algorithmic}
\end{algorithm}

The varieties of tensors that admit such a decomposition are data loci in the sense of~\cite{horobect2022data}. We denote by $\DL_r$ the variety of tensors that admit such a decomposition of length $r$. 
We obtain a chain of varieties $\DL_1\subseteq \DL_2\subseteq\ldots$, such that any tensor $\T\in\DL_r$ has a rank-one approximation $x$ with $\T-x\in \DL_{r-1}$. 
For symmetric tensors, this chain of varieties stabilizes and the limit is the variety of weakly-odeco tensors~\cite{horobet2023does}.
For matrices, $\DL_r$ is the variety of matrices of rank at most $r$, so the chain also stabilizes, filling the ambient space.
In general, the rank-one tensor $x_i$ need not be a singular vector tuple of $\T$ for $i \geq 2$. That is, the decomposition is order-dependent. 

\begin{Example}\label{ex:order-dependent}
Let $\{e_0, e_1\}$ be an orthonormal basis of $\R^2$, and fix
\[
\T = \underbrace{e_1 \otimes (e_0 + e_1) \otimes (e_0 - e_1)}_{x_1} + \underbrace{e_0 \otimes e_1 \otimes e_1}_{x_2} + \underbrace{e_0 \otimes e_0 \otimes e_0}_{x_3} \in (\R^2)^{\otimes 3}.
\]
We observe that $\T \in \DL_3$, and that $x_1$ is a critical rank-one approximation of $\T$, but $x_2$ and $x_3$ are not critical rank-one approximations of~$\T$.
\end{Example}
We consider subvarieties of $\DL_r$ consisting of tensors for which the decomposition does not depend on the order. These are the tensors 
$\T = \sum_{i=1}^r x_i$
in which, for any order of the summands, $x_j$ is a critical rank-one approximation of~$\T - \sum_{i < j}x_i$. In particular, every~$x_i$ is a critical rank-one approximation of the original tensor $\T$. Put differently, for these tensors \Cref{alg:1} terminates with $\mathcal{S}$ replaced by $\T$ in Line~4. While a generic tensor lies in the span of its singular vector tuples~\cite{DOT}, one does not know the coefficient of each singular vector tuple; i.e. how much of each one to subtract off. For the tensors we consider, the coefficient of each singular vector tuple is its singular value, so the tensor is a sum of critical rank-one approximations.

One class of tensors known to possess such a property are the odeco (orthogonally decomposable) tensors, those that have a decomposition
$\T = \sum_{i = 1}^r x_i^{(1)} \otimes \cdots \otimes x_i^{(d)}$ such that for all~$i \neq j$
we have the orthogonality $x_{i}^{(k)} \perp x_{j}^{(k)}$ in all factors $1 \leq k \leq d$ \cite{robeva2016odeco, robeva2017singular}.  We relax this notion by requiring only that for all $i\neq j$, the $i$-th and $j$-th summands are orthogonal in at least two factors. We call such a decomposition \emph{two-orthogonal}. We define~$W_r$ to be the Zariski closure of the set of tensors that have a two-orthogonal decomposition of length at most $r$. The \emph{two-orthogonal variety} is $W = \overline{\bigcup_{r \geq 1} W_r}$. See \Cref{fig:two-orthogonal-plot}. 

\begin{figure}[ht]
    \centering
    \includegraphics[width=0.45\linewidth, height=0.35\linewidth]{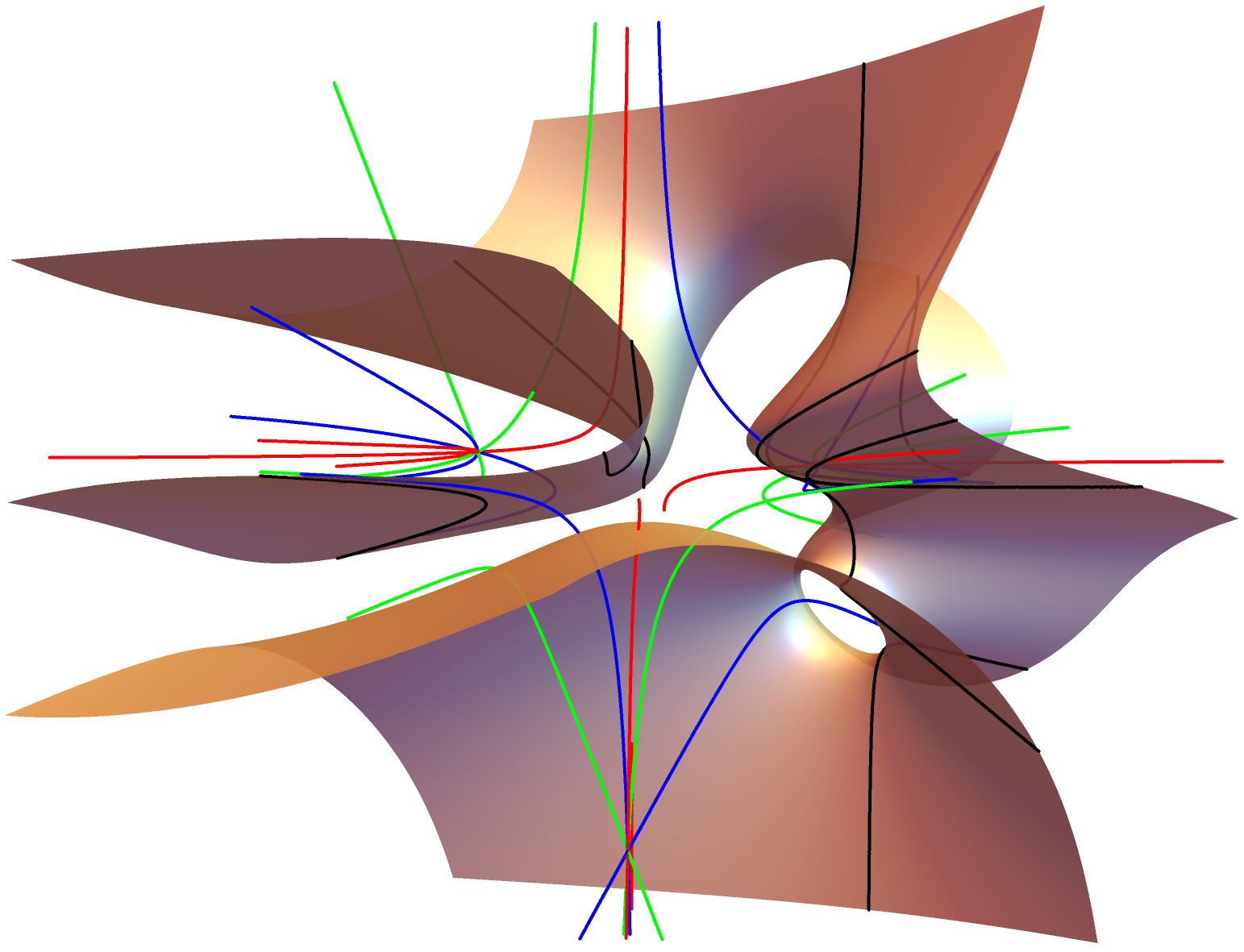}
    \hfill
    \includegraphics[width=0.45\linewidth, height=0.35\linewidth]{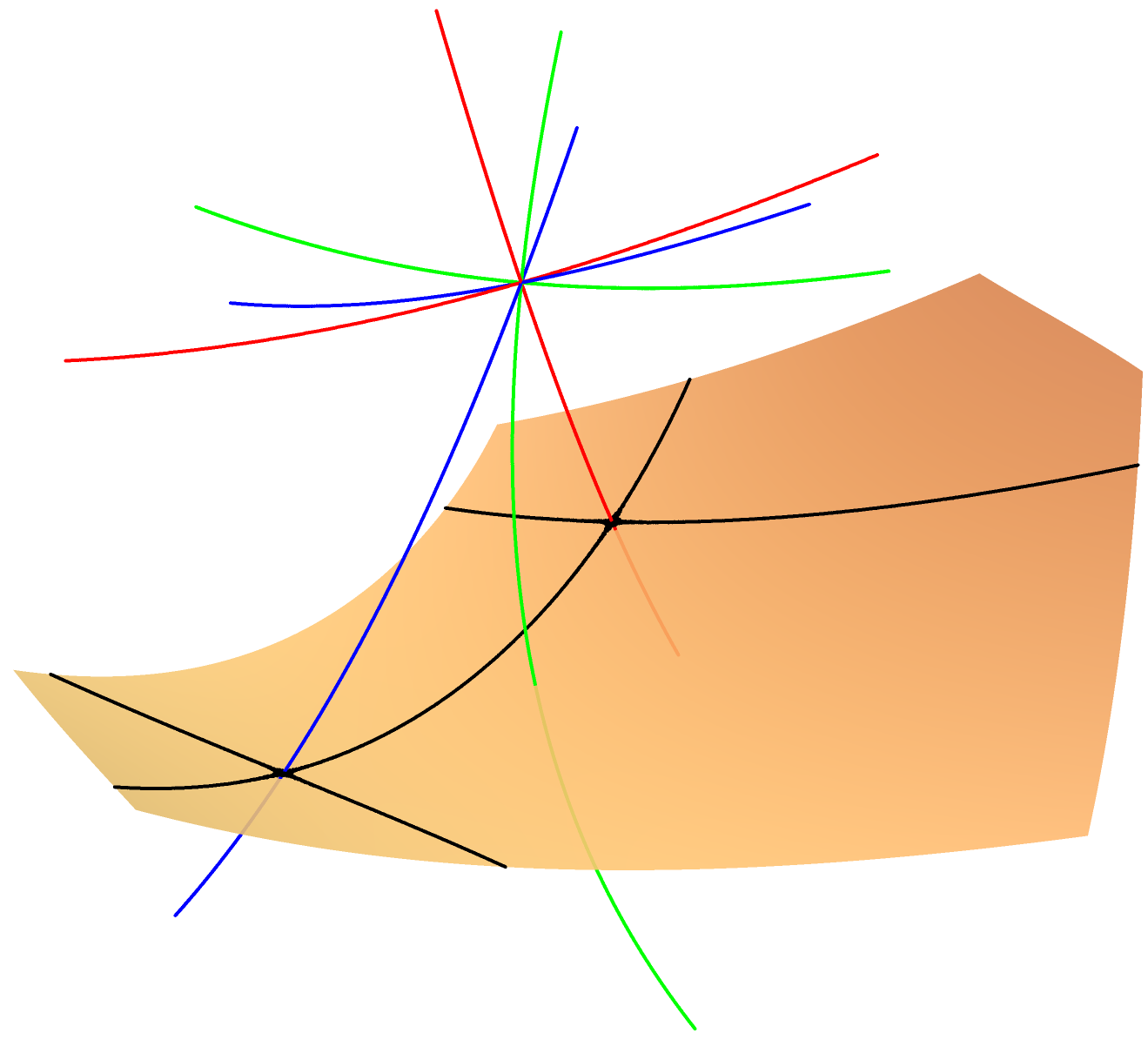}
    \caption{The affine slice of the two-orthogonal variety in $\R^2 \otimes \R^2 \otimes  \R^2$ when~$t_{010}=1, t_{100}=2t_{000}, t_{110}=1,t_{101}=2, t_{011}=3$. The image on the right is a zoomed-in and rescaled version of the image on the left. The surface is~$W_4 \setminus W_3$. The black curve is $W_3 \setminus W_2$. The red, green, and blue curves are the three components of $W_2$. The intersection of these three curves is the odeco variety. These plots were made with \texttt{Mathematica} \cite{Mathematica}.}
    \label{fig:two-orthogonal-plot}
\end{figure}

The notion of a two-orthogonal tensor also appears in~\cite[Definition 3.2]{vannieuwenhoven2014generic}, where it is called weakly two-orthogonal. We drop `weakly' to avoid confusion with the weakly orthogonally decomposable tensors, which appear in~\cite{horobet2023does}. Two-orthogonality is necessary for having a Schmidt--Eckart--Young decomposition~\cite[Theorem 3.3]{vannieuwenhoven2014generic}. The singular value decomposition implies that every matrix is two-orthogonal. For symmetric decompositions, two-orthogonality implies orthogonality in all factors, so it leads to an odeco tensor. We will see that a generic $2\times 2 \times 2$ tensor is not two-orthogonal in \Cref{sec:2x2x2}. We believe this is the general behavior for tensors of order at least three, as two-orthogonality imposes closed conditions on the singular vector tuples.

Now we state our main results. First, we show that the two-orthogonal tensors are the tensors that can be decomposed via rank-one approximations.

\begin{restatable}{Theorem}{nonorderdep}\label{thm:non_order_dep}
    The set of two-orthogonal tensors in $\R^{n_1} \otimes \cdots \otimes \R^{n_d}$ coincides with the set of tensors with a decomposition $\T = \sum_{i=1}^r x_i$ such that each $x_j$ is a critical rank-one approximation of $\T - \sum_{i \in \sI} x_i$ for all $\sI \subseteq \{1, \dots, r\} \setminus \{ j\}$.
\end{restatable}

Two-orthogonal decompositions cannot have arbitrarily many summands. We count the maximum number of possible summands in the following.

\begin{restatable}{Theorem}{maximallength}   
\label{thm:maximal_length}
    The maximal length of a two-orthogonal decomposition in $\R^{n_1} \otimes \cdots \otimes \R^{n_d}$ is~$N=\min_{1 \leq k \leq d} \prod_{j \neq k} n_j$. In particular, the two-orthogonal variety is $W = W_N$.
\end{restatable}

By constructing two-orthogonal decompositions of maximal length, we parameterize a family of two-orthogonal tensors with a dimension-preserving map. This gives a lower bound on the dimension of the two-orthogonal variety.

\begin{restatable}{Theorem}{dimensionlowerbound} \label{thm:dimension_lower_bound}
    The dimension of the two-orthogonal variety in $(\mathbb{R}^n)^{\otimes d}$ is at least
    \[
    n^{d-1} + d \binom{n}{2}.
    \]
\end{restatable}

In $V = (\R^n)^{\otimes d}$, the odeco variety has dimension $n + d \binom{n}{2}$. Hence, the dimension of the two-orthogonal variety exceeds that of the odeco variety by at least $n(n^{d-2}-1)$.
Finally, we show that two-orthogonal decompositions are generically identifiable for $2\times 2 \times 2$ tensors. We suspect that this is also true for generic two-orthogonal tensors in $\R^{n_1} \otimes \cdots \otimes \R^{n_d}$.

\begin{restatable}{Theorem}{uniqueness} \label{thm:uniqueness}
    A generic tensor in $W \subseteq (\R^{2})^{\otimes 3}$ has a unique two-orthogonal decomposition.
\end{restatable}

The article is organized as follows. In \Cref{sec:2o-decompositions} we study the main properties of two-orthogonal decompositions. We prove \Cref{thm:non_order_dep} and \Cref{thm:maximal_length}, we introduce the notion of two-orthogonal rank and how it relates to the usual rank, and we extend some of our results to partially symmetric tensors. We also show that truncating two-orthogonal decompositions may lead to best low-rank approximations and answer an open question posed in \cite{vannieuwenhoven2014generic}. In \Cref{sec:two-orthogonal-variety} we prove \Cref{thm:dimension_lower_bound}. In \Cref{sec:combinatorial_descriptions} we provide a parametric description of the two-orthogonal variety using graphs, with a focus on binary tensors. There is an interplay between algebraic geometry and combinatorics, with open directions for future work. Finally, in \Cref{sec:2x2x2} we provide an algebraic description of the two-orthogonal variety in $(\R^2)^{\otimes 3}$ and prove \Cref{thm:uniqueness}.

\section{Two-orthogonal decompositions} \label{sec:2o-decompositions}
Let $V = \R^{n_1} \otimes \cdots \otimes \R^{n_d}$ with $d\geq 2$ and each $n_k \geq 2$. Let $X = \{ v_1 \otimes \cdots \otimes v_d \mid v_k \in \R^{n_k}\} \subset V$ denote the cone over the Segre variety, the set of rank-one tensors. Let $\langle\cdot,\cdot \rangle_k$ be nondegenerate bilinear forms in~$\R^{n_k}$ for $k=1, \dots, d$. They induce a nondegenerate bilinear form in V, called the \emph{Bombieri-Weyl inner product}, defined on two rank-one tensors as
\[
\langle v_1\otimes\cdots\otimes v_d,w_1\otimes\cdots\otimes w_d \rangle=\prod_{k=1}^d\langle v_k,w_k \rangle_k
\]
and extended to all of $V$ by bilinearity. We write $v_k \!\perp\! w_k$ to denote orthogonality: $\langle v_k ,\! w_k\rangle_k \!=\! 0$.
Given a positive integer $n$, let $[n] = \{1, \dots, n \}$, and let $\{e_1, \dots, e_n\}$ be an orthonormal basis of $\R^n$. We use the same letters for the bases of different factors of $V$. That is, the set~$\{ e_{i_1} \otimes \cdots \otimes e_{i_d} \mid i_k \in [n_k]\}$ is an orthonormal basis of~$V$. Let $\| \cdot \|$ denote the Frobenius norm on $V$ induced by the inner product.
\begin{Definition}
\label{def:wo2}
A \emph{two-orthogonal decomposition} of $\T$ is an expression $\T = \sum_{i = 1}^r x_i$, where $x_i = x_i^{(1)} \otimes \cdots \otimes x_i^{(d)} \neq 0$ and all pairs of summands $x_i$ and $x_j$ with $i \neq j$ are orthogonal in at least two factors: there exist indices $k_1\neq k_2$ such that $x_{i}^{(k_1)} \perp x_{j}^{(k_1)}$ and $x_{i}^{(k_2)} \perp x_{j}^{(k_2)}$. A tensor $\T$ is \emph{two-orthogonal} if it has a two-orthogonal decomposition for some $r$.
\end{Definition}

\begin{Remark}\label{remark:pythagoras}
    The summands in a two-orthogonal decomposition are orthogonal to each other. Hence, a two-orthogonal decomposition $\T = \sum_{i=1}^r x_i$ has $\|\T\|^2=\sum_{i=1}^r \|x_i\|^2$. 
\end{Remark}

\begin{Definition}\label{def:SVT}
Given a tensor $\T\in V$, a
rank-one tensor $x=x^{(1)} \otimes \cdots \otimes x^{(d)}$ is a \emph{singular vector tuple} of $\T$ if for all $k \in [d]$ there exists $\lambda_k \in \R$ such that
\[
\T(x^{(1)}, \dots, x^{(k-1)}, \cdot, x^{(k+1)}, \dots, x^{(d)}) = \lambda_k x^{(k)}
\]
where $\T$ is viewed as a multilinear map and $\R^{n_k}$ is identified with its dual $(\R^{n_k})^\ast$ via $\langle\cdot,\cdot\rangle_k$.
If $\langle x^{(k)},x^{(k)}\rangle_k = 1$ for all $k$, then $\lambda_k = \lambda$ for all $k$, and $\lambda$ is called the \emph{singular value} of $x$.
\end{Definition}

\begin{Remark}
 We consider singular vector tuples defined over $\R$, though they are usually defined over $\CC$. We focus on $\R$ because approximating tensors by rank-one tensors is done primarily on the real numbers in applications. To define them over $\CC$, one considers the bilinear form $\langle \cdot,\cdot\rangle_k$ as an extension of a real inner product to a complex product instead of opting for a Hermitian inner product. However, such extensions introduce isotropic vectors: $v\in\CC^{n_k} \setminus \{0\}$ with~$\langle v,v\rangle_k=0$. While a general tensor has no isotropic singular vector tuples~\cite[Proposition 2.6]{DOT}, our tensors are special. In fact, self-orthogonality plays a role in the study of symmetric odeco tensors~\cite{horobet2023does}. We leave the study of two-orthogonal tensors in $\CC^{n_1} \otimes \cdots \otimes \CC^{n_d}$ for future work.
\end{Remark}

We view singular vector tuples as points in the corresponding Segre product \cite{FO14}. Hence, the normalization $\langle x^{(k)}, x^{(k)} \rangle_k = 1$ is without loss of generality, although it is necessary for the singular value to be well defined. Let $\T = \sum_{i=1}^r x_i$ be a two-orthogonal decomposition. Then, each $x_i$ is a singular vector tuple for $\T$ and, after normalizing each factor, the corresponding singular value is $\|x_i\|$. We recall the following facts.

\begin{Proposition} \label{prop:tangent-space}
    Fix $\T \in V$ and $x = x^{(1)} \otimes \cdots \otimes x^{(d)} \in X \setminus \{ 0\} \subset V$.
    \begin{enumerate}
        \item The tangent space to $X$ at $x$ is $T_xX=\sum_{k=1}^d x ^{(1)}\otimes\cdots\otimes x^{(k-1)} \otimes \R^{n_k}\otimes x^{(k+1)} \otimes \cdots \otimes x^{(d)}$.
        \item Suppose that $\langle x^{(k)}, x^{(k)} \rangle_k =1$ for all $k \in [d]$, then $x$ is a singular vector tuple of $\T$ with singular value $\lambda$ if and only if $T-\lambda x \in N_xX = (T_xX)^\perp$, the normal space to $X$ at $x$.
        \item The singular vector tuples correspond to the critical points of the distance function $d_\T(y):=\langle \T - y, \T-y\rangle$, on $y \in X$. A singular vector tuple $x$ with nonzero singular value~$\lambda$ corresponds to the critical rank-one approximation $\lambda x$, and vice versa.
    \end{enumerate}
\end{Proposition}
\begin{proof}
Part (1) follows by applying the Leibniz rule to the parametrization of $X$. For part~(2), the equations in \Cref{def:SVT} are equivalent to
\[
\langle \T,x^{(1)}\otimes\dots\otimes x^{(k-1)}\otimes v^{(k)}\otimes x^{(k+1)}\otimes\dots \otimes x^{(d)}\rangle=\lambda \langle v^{(k)},x^{(k)}\rangle_k,
\]
for all $v^{(k)} \in \R^{n_k}$. Since $\langle x^{(k)}, x^{(k)} \rangle_k = 1$ for all $k \in [d]$, this can be rewritten as
\[
\langle \T- \lambda x,x^{(1)}\otimes\dots\otimes x^{(k-1)}\otimes v^{(k)}\otimes x^{(k+1)}\otimes\dots \otimes x^{(d)}\rangle=0
\]
for all $v^{(k)} \in \R^{n_k}$, so the statement follows. Part (3) follows from (2), see~\cite[\S 3]{lim2005singular}.
\end{proof}

The following proposition shows that every tensor can be approximated as a sum of critical rank-one approximations: subtracting off best rank-one approximations decreases the norm of a tensor, and the norm of the residual tends to zero in the limit. This has been studied before~\cite{falco2011proper, qi_2011}, but we will provide a proof for completeness.

\begin{Proposition}
	Fix a tensor $\T \in \R^{n_1} \otimes \cdots \otimes \R^{n_d}$. For each $j \in \NN$, let $x_j$ be a best rank-one approximation of $\T - \sum_{i < j}x_i$. Then $\sum_{i = 1}^n x_i \to \T$ as $n \to \infty$. 
\end{Proposition}
\begin{proof}
We show that $\| \T - \sum_{i=1}^n x_i \| \to 0$ as $n \to \infty$.  Let $\| \cdot \|_\sigma$ denote the the spectral norm. 
We have $\| x_1 \| = \| \T \|_\sigma$, by the definition of best rank-one approximation.
Moreover, $\|\T \|^2 = \| \T - x_1 \|^2 + \| x_1 \|^2$,
since a critical rank-one approximation $x_1$ and its residual $\T - x_1$ are orthogonal.
Hence
\[ \| \T - x_1 \|^2 = \| \T \|^2 - \| \T \|_\sigma^2 \leq (1 - c^2) \|\T \|^2 ,\]
where the constant $c >0$ relates the spectral and Frobenius norms; i.e. $c\| \sS\| \leq \| \sS\|_\sigma$ for all tensors $\sS \in \R^{n_1} \otimes \cdots \otimes \R^{n_d}$. This constant exists because two norms in a finite-dimensional vector space are equivalent. We have $0 < c < 1$, since $ \| \sS\|_\sigma < \| \sS\|$ if $\sS$ is not a rank-one tensor.
Repeating the above for iterated best rank-one approximations yields 
\[
\| \T - \sum_{i \leq j} x_i\|^2 \leq (1 - c^2) \| \T - \sum_{i < j} x_i\|^2 \leq \cdots \leq (1 - c^2)^j \| \T\|^2.
\] 
Since $0 < c < 1$, we have $ \| \T - \sum_{i \leq n} x_i\| \to 0$ as $n \to \infty$. 
\end{proof}

The previous result shows that one can use \Cref{alg:1} for \emph{any} input tensor $\T$ and get arbitrarily good approximations, from a numerical perspective. We study when these decompositions are exact, not just arbitrarily good. Moreover, as shown in \Cref{ex:order-dependent}, these decompositions may be order-dependent. We are interested in decompositions that are valid independent of the order. The following result characterizes these decompositions.

\nonorderdep*
\begin{proof}
    For a rank-one tensor $x = x^{(1)} \otimes \cdots \otimes x^{(d)} \in X \subset \R^{n_1} \otimes \cdots \otimes \R^{n_d}$ we have that $x$ is a critical rank-one approximation of a tensor~$\T$ if and only if $\T-x \in N_xX$, and $T_xX = \sum_{k=1}^d x^{(1)} \otimes \cdots \otimes \R^{n_k} \otimes \cdots \otimes x^{(d)}$, see \Cref{prop:tangent-space}.  
    
    If $\T =\sum_{i=1}^r x_i$ is two-orthogonal, then for all $i \neq j$ we have $x_i \in N_{x_j}X$. Therefore, any linear combination $\sum_{i \neq j} \alpha_i x_i$ lies in $N_{x_j}X$. Hence, $x_j$ is a critical rank-one approximation of $\T - \sum_{i \in \sI} x_i$ for all $\sI \subseteq \{1, \dots, r\} \setminus \{ j\}$.

    Conversely, consider a tensor $\T = \sum_{i=1}^r x_i$ such that $x_j$ is a critical rank-one approximation of $\T - \sum_{i \in \sI} x_i$ for all $\sI \subseteq \{1, \dots, r\} \setminus \{ j\}$. Then, for any $i \neq j$, $x_j$ is a critical rank-one approximation of $x_i + x_j$. Hence, $x_i \in N_{x_j}X$. So, for all $k \in [d]$ and all $v_k \in \R^{n_k}$ we have
    \[
    \langle x_i^{(1)} \otimes \cdots \otimes x_i^{(d)}, x_j^{(1)} \otimes\cdots \otimes v^{(k)} \otimes \cdots \otimes x_j^{(d)}\rangle = \langle x_i^{(k)}, v^{(k)} \rangle_k \prod_{l \neq k} \langle x_i^{(l)}, x_j^{(l)}\rangle_l = 0.
    \]
    These equations are satisfied if and only if there exist two indices $k_1 \neq k_2 \in [d]$ such that~$x_i^{(k_1)} \perp x_j^{(k_1)}$ and $x_i^{(k_2)} \perp x_j^{(k_2)}$. This holds for all pairs of summands, so $\T$ is a two-orthogonal tensor.
\end{proof}

\begin{Definition}\label{def:W_r}
    Define $W_r$ to be the variety of length-at-most-$r$ two-orthogonal tensors
    \[
    W_r=\overline{\left\{ \sum_{i=1}^{r} x_i^{(1)}\otimes\dots\otimes x_i^{(d)} \mid x_{i}^{(k_1)}\perp x_{j}^{(k_1)}, x_{i}^{(k_2)}\perp x_{j}^{(k_2)}\text{ for at least two indices } k_1\neq k_2\right\}}.
    \]
    We allow zero summands here, so the length can be smaller than $r$. For $r=1$, no conditions are imposed, so $W_1 = X$ is the set of rank-one tensors.
    The overline denotes Zariski closure. We define the \emph{two-orthogonal variety} to be $W = \overline{\bigcup_{r} W_r}$.
\end{Definition}

\begin{Remark}
    The two-orthogonal variety $W \subseteq V = \R^{n_1} \otimes \cdots \otimes \R^{n_d}$ is invariant under the action of the product of special orthogonal groups $G = \SO(n_1)\times \cdots \times \SO(n_d)$ on $V$, acting by change of basis in each factor.
\end{Remark}

\begin{Remark}\label{rk:equations}
    The equations defining some components of the two-orthogonal variety relate to the odeco equations \cite[Theorem 9]{boralevi2017orthogonal}. For example, consider a two-orthogonal tensor $\T = x_1 + x_2 \in W_2 \subset (\R^2)^{\otimes d}$ such that~$x_1$ and $x_2$ are orthogonal in the $k_1$-th and~$k_2$-th factors, then $\T \star_{k_i} \T \in \bigotimes_{l \neq k_i} S^2(\R^2)$ for $i = 1,2$, where $\star_{k_i}$ denotes contraction along the $k_i$-th factor and $S^2(\R^2)$ denotes the space of $2 \times 2$ symmetric matrices. One can check that these equations define the components of $W_2$ for~$(\R^2)^{\otimes 3}$ (see \Cref{sec:2x2x2}) and $(\R^2)^{\otimes 4}$. A similar argument may be used for some components of $W_n \subset (\R^n)^{\otimes d}$. Finding the implicit equations that describe the two-orthogonal variety is left to future work.
\end{Remark}

\subsection{Maximal length of two-orthogonal decompositions.} From the definition of $W_r$, we have a chain of varieties $X = W_1 \subseteq W_2 \subseteq \cdots \subseteq W_r \subseteq \cdots$. The following result shows that this chain stabilizes in finitely many steps. That is, there exists $N \in \NN$ such that~$W_N = W_{N+1} = \cdots$.

\begin{Lemma}\label{lemma:stabilization}
    A two-orthogonal decomposition of a tensor in $\mathbb{R}^{n_1} \otimes \cdots \otimes \mathbb{R}^{n_d} $ has at most $N=\min_{k \in [d]} \prod_{j \neq k} n_j$ summands. In particular, the two-orthogonal variety is $W = W_N$.
\end{Lemma}

\begin{proof}
    Fix $\T = x_1 + \cdots + x_r \in W_r \subset \R^{n_1} \otimes \cdots \otimes \R^{n_d}$, where $x_i = x_i^{(1)} \otimes x_i^{(2)} \otimes \cdots \otimes x_i^{(d)} \neq 0$ for all $i \in [r]$ and these summands are orthogonal in at least two factors. Without loss of generality, suppose that $n_1 \leq \cdots \leq n_d$. Consider the tensor $\tilde{\T} = \tilde{x}_1 + \cdots + \tilde{x}_r \in \R^{n_1} \otimes \cdots \otimes \R^{n_{d-1}}$, where $\tilde{x}_i = x_i^{(1)}  \otimes \cdots \otimes x_i^{(d-1)}$ for all $i \in [r]$.
    The tensors $\tilde{x}_i$ and $\tilde{x}_j$ are orthogonal (in at least one factor) for all $i \neq j$, since $x_i$ and $x_j$ are orthogonal in at least two factors. In particular, $\tilde{x}^{(1)}, \dots, \tilde{x}^{(r)}$ are linearly independent. Indeed, if we had $\tilde{x}_i = \sum_{j \neq i} \alpha_j \tilde{x}_j$, taking the inner product with~$\tilde{x}_i$ leads to $\| \tilde{x}_i\|^2 = 0$, a contradiction. In conclusion,
    \[
    r = \dim \left(\mathrm{span}\{ \tilde{x}^{(1)}, \dots, \tilde{x}^{(r)} \} \right) \leq \dim\left( \R^{n_1} \otimes \cdots \otimes \R^{n_{d-1}} \right) = n_1n_2 \cdots n_{d-1}. \qedhere
    \]
\end{proof} 

\begin{Definition}[{see~\cite{mckay2008hypercubes}}]
A \emph{Latin square} $L$ is an $n \times n$ matrix with elements in $[n]$ such that every row and every column of $L$ are a permutation of $[n]$. Put differently, for every $i \in [n]$  the maps $k \mapsto L(i,k)$ and $k \mapsto L(k,i)$ are permutations of $n$. More generally, a \emph{Latin hypercube} is an array $L$ indexed by $[n]^d$ satisfying the following: for every $j \in [d]$ and every $(i_1, \dots, i_{j-1}, i_{j+1}, \dots, i_d) \in [n]^{d-1}$ the map $k \mapsto L(i_1, \dots, i_{j-1}, k, i_{j+1}, \dots, i_{d})$ is a permutation of $[n]$. 
\end{Definition}

There exists a Latin hypercube for every $n$ and $d$, e.g. $L(i_1, \dots , i_{d}) = i_1 + \dots + i_{d} \mod n$, where the sum is understood to be an element in $[n]$. A Latin square $L$ gives two-orthogonal tensors of order three, as follows. Consider the set $\sI = \{ (i,j, L(i,j)) \mid i,j \in [n]\} \subset [n]^3$. We can identify each tuple $\bi = (i_1, i_2, i_3) \in \sI$ with the rank-one tensor $e_{i_1} \otimes e_{i_2} \otimes e_{i_3}$. Every pair of distinct tuples $\bi, \bj \in \sI$ differs in at least two indices. Therefore, the tensors $e_{i_1} \otimes e_{i_2} \otimes e_{i_3}$ and $e_{j_1} \otimes e_{j_2} \otimes e_{j_3}$ are orthogonal in at least two factors.

\begin{Example} Let $V = (\R^3)^{\otimes 3}$. A $3 \times 3$ Latin square corresponds to a two-orthogonal tensor in $V$ with 9 summands. For example,
\begin{center}
    \vspace{0.5em}
    \begin{tabular}{|c|c|c|cl}
        \cline{1-3} $1$& $2$& $3$ & & $e_1 \otimes e_1 \otimes e_1 + e_1 \otimes e_2 \otimes e_2 + e_1 \otimes e_3 \otimes e_3 +$ \\
        \cline{1-3} $2$&$3$ &$1$ & $\leftrightarrow$ & $e_2 \otimes e_1 \otimes e_2 + e_2 \otimes e_2 \otimes e_3 + e_2 \otimes e_3 \otimes e_1 +$ \\
        \cline{1-3} $3$&$1$ &$2$ & & $e_3 \otimes e_1 \otimes e_3 + e_3 \otimes e_2 \otimes e_1 + e_3 \otimes e_3 \otimes e_2$. \\
        \cline{1-3}
        \end{tabular}
        \vspace{0.5em}
\end{center}
    \noindent \Cref{lemma:stabilization} implies that the maximum number of summands of a two-orthogonal tensor in $V$ is nine. Each summand can be scaled arbitrarily and two-orthogonality is preserved.
\end{Example}

\begin{Lemma}\label{lemma:construction_maximal}
    There exists a two-orthogonal decomposition in $V = \R^{n_1} \otimes \cdots \otimes \R^{n_d}$ with $N = \min_{k \in [d]} \prod_{j \neq k} n_j$ summands.
\end{Lemma}
    
\begin{proof}
    We construct a two-orthogonal decomposition of maximal length using Latin hypercubes. First, suppose $n_1 = \cdots = n_d = n$. Let $L$ be a Latin hypercube indexed by $[n]^{d-1}$. Consider the set $\sI = \{(i_1, \dots, i_{d-1}, L(i_1, \dots , i_{d-1})) \mid i_k \in [n] \}$. Then~$|\sI|= n^{d-1}$ and every pair of tuples in~$\sI$ differ in at least two indices. This gives the family of two-orthogonal tensors in~$(\R^n)^{\otimes d}$:
\[
\sum_{\bi \in \sI} \lambda_{\bi} e_{i_1} \otimes \cdots \otimes e_{i_{d}}
\]
for any $\lambda_\bi \in \R$. Let $V = \R^{n_1} \otimes \cdots \otimes \R^{n_d}$, with $n_1 \leq \dots \leq n_d$. Consider a Latin hypercube indexed by $[n_{d-1}]^{d-1}$ and choose a subarray of format $n_1 \times \cdots \times n_{d-1}$. This gives a two-orthogonal decomposition of length $n_1 \cdots n_{d-1}$ in $V$, following the same reasoning.
\end{proof}

\maximallength*
\begin{proof}
\Cref{lemma:stabilization} implies that a two-orthogonal decomposition cannot have more than $N = \min_{k \in [d]} \prod_{j \neq k} n_j$ summands. \Cref{lemma:construction_maximal} shows how to construct two-orthogonal decompositions with $N$ summands.
\end{proof}

Decompositions obtained from Latin hypercubes satisfy the following property.

\begin{Definition}\label{def:basis-aligned}
    A two-orthogonal decomposition $\T = \sum_{i=1}^{r} x_i^{(1)} \otimes \cdots \otimes x_i^{(d)}$ is
    \emph{basis-aligned} if for all $k \in [d]$ and all $i,j \in [r]$, $x_i^{(k)}$ and $x_j^{(k)}$ are either collinear or orthogonal.
\end{Definition}

Not all two-orthogonal decompositions are basis-aligned. For example, consider the decomposition $\T = e_1 \otimes e_1 \otimes e_1 + (e_1 + e_2) \otimes e_2 \otimes e_2 \in W_2$. We explore two-orthogonal decompositions that are not basis-aligned in \Cref{sec:combinatorial_descriptions}.

\subsection{Two-orthogonal rank} Two-orthogonal decompositions are not unique, in general.
\begin{Example}\label{ex:non-unique-decomposition}
    Let $d \geq 2$, let $V = (\R^2)^{\otimes d}$, and let $\sI$ be the set of binary strings of length $d$ with an even number of ones: $\mathcal{I} = \{ (i_1, \dots, i_d) \in \{0,1 \}^d  \mid \sum_{k=1}^d i_k = 0 \mod 2\}$. Let $\{e_0, e_1\}$ be an orthogonal basis of $\R^2$. The following tensor admits two different two-orthogonal decompositions:
    \[
    \T = \sum_{(i_1, \dots, i_d) \in \sI} e_{i_1} \otimes \cdots \otimes e_{i_d} = \frac{1}{2} (e_0 + e_1)^{\otimes d} + \frac{1}{2}(e_0 - e_1)^{\otimes d}.
    \]
    The first decomposition has maximal length, but the second one shows that $\T \in W_2$.
\end{Example}
While \Cref{lemma:construction_maximal} constructs two-orthogonal decompositions of maximal length, it is unclear whether these tensors may have two-orthogonal decompositions with fewer terms. We do not know whether the stabilization of the chain $W_1 \subseteq W_2 \subseteq \cdots $ occurs exactly in $N = \min_{k \in [d]} \prod_{j \neq k} n_j$ steps or before. That is, we do not know whether~$W_{N-1} = W_N$. 
The previous example motivates the following definition.

\begin{Definition}
    The {\em two-orthogonal rank} of a two-orthogonal tensor $\T \in \mathbb{R}^{n_1} \otimes \cdots \otimes \mathbb{R}^{n_d}$ is the smallest number $r$ such that $\T$ admits a two-orthogonal decomposition of length $r$. 
\end{Definition}

The rank of a tensor $\T$ is the smallest number $r$ such that $\T$ can be expressed as the sum of $r$ rank-one tensors. 
The two-orthogonal rank and rank coincide for matrices, due to the singular value decomposition. For higher-order tensors, we can only say that the two-orthogonal rank of $\T$ is at least the rank of $\T$. We will see in \Cref{sec:2x2x2} that for $2\times 2 \times 2$ tensors the maximum two-orthogonal rank is 4 while the maximum rank is 3.

The notions of rank and two-orthogonal rank could potentially coincide up to the maximal rank. For example, in $(\R^n)^{\otimes 3}$, the maximal rank is at most $\frac{1}{2}n(n+2)-1$ \cite[Theorem 3.4]{sumi2010maximal}, while the maximal potential two-orthogonal rank is $n^2$. \Cref{prop:rankWr} shows that this is the case for small enough ranks.
The \emph{border rank} of a tensor $\T$ is the smallest $r$ such that $\T = \lim_{\epsilon \to 0} \T_\epsilon$ where each $\T_\epsilon$ has rank~$r$, and it may be smaller than the rank of $\T$ (see \Cref{ex:rank222}).

\begin{Proposition} \label{prop:rankWr}
    In $(\R^{n})^{\otimes d}$,  $W_{r-1} \subsetneq W_{r}$ if $r \leq n$. Moreover, a general tensor in $W_r \setminus W_{r-1}$ with $r \leq n$ has border rank $r$, so it has rank $r$.
\end{Proposition}
\begin{proof}
     After an orthogonal change of basis, an odeco tensor of two-orthogonal rank $r$ is of the form $\T = \sum_{i=1}^{r} \lambda_i e_i^{\otimes d}$ for $\lambda_i \in \R$. This tensor has border rank $r$, e.g. by looking at the flattenings \cite{landsberg2011tensors}. Therefore, $W_{r-1} \neq W_r$. A property holds for a general tensor if it holds on a dense open set, and odeco tensors of rank $r$ lie in the intersection of all the irreducible components of $\overline{W_r \setminus W_{r-1}}$.
\end{proof}

This result can be improved by considering the following class of two-orthogonal tensors. A decomposition $\T = \sum_{i=1}^r x_i^{(1)} \otimes \cdots \otimes x_i^{(d)} \in \R^{n_1} \otimes \cdots \otimes \R^{n_d}$ is called \emph{strong two-orthogonal} (\cite[Definition 3.5]{vannieuwenhoven2014generic}) if the orthogonalities between summands always occur in the same partition of the factors: there exists a nonempty set $\sJ \subsetneq [d]$ such that for all $i\neq j$ we have  $x_{i}^{(k_1)} \perp x_{j}^{(k_1)}$ and~$x_{i}^{(k_2)} \perp x_{j}^{(k_2)}$ for some indices $k_1 \in \sJ$ and $k_2 \in [d] \setminus \sJ$. Put differently, the decomposition is still two-orthogonal when viewed as a matrix decomposition in~$\left(\bigotimes_{j \in \sJ} \R^{n_j} \right) \otimes \left(\bigotimes_{j \not\in \sJ} \R^{n_j} \right)$.

\begin{Proposition} \label{prop:bound-strong-two-orthogonal}
    Let $V = \R^{n_1} \otimes \cdots \otimes \R^{n_d}$ and let $m = \max_{\sJ \subsetneq [d]} \min \{ \prod_{j \in \sJ} n_j, \prod_{j \not\in \sJ} n_j\}$. Then, for all $r \leq m$, a sufficiently general tensor of $W_r$ has border rank $r$. In particular, we have that $W_{m-1} \subsetneq W_m$.
\end{Proposition}
\begin{proof}
    By sufficiently general we mean that it holds for some components of~$W_r \setminus W_{r-1}$. Consider a nonempty $\sJ \subsetneq[d]$. Picking orthogonal basis for $\bigotimes_{j \in \sJ} \R^{n_j}$ and $\bigotimes_{j \not\in \sJ} \R^{n_j}$ we can construct a strong two-orthogonal decomposition of length $\min \{ \prod_{j \in \sJ} n_j, \prod_{j \not\in \sJ} n_j\}$.
    By flattening the tensor into a matrix, the singular value decomposition implies that strong two-orthogonal decompositions of length $r$ lead to tensors of border rank $r$, see \cite[Theorem 3.6]{vannieuwenhoven2014generic}.
\end{proof}
\begin{Example}
    When $V = (\R^{n})^{\otimes d}$, \Cref{prop:bound-strong-two-orthogonal} implies that $W_{r-1} \subsetneq W_r$ if $r \leq n^{\lfloor \frac{d}{2}\rfloor}$.
\end{Example}

\begin{Conjecture}\label{conj:rank}
    Let $r_g$ be the generic rank in $\CC^{n_1} \otimes \cdots \otimes \CC^{n_d}$. For every $r \leq r_g$, a sufficiently general tensor $\T \in W_r \subset \R^{n_1} \otimes \cdots \otimes \R^{n_d}$ has rank $r$.
\end{Conjecture}
The following examples provide evidence for this conjecture. 
One can compute the rank of the tensors by applying the technique used in \cite{Wang_2023}, which is illustrated in \Cref{ex:rank-proved}.

\begin{Example} \label{ex:rank222}
    The two-orthogonal tensor
    \[
    \T = e_1 \otimes e_1 \otimes e_2 + e_1 \otimes e_2 \otimes e_1 + e_2 \otimes e_1 \otimes e_1 \in W_3 \subset (\R^2)^{\otimes 3}
    \]
    has rank 3, which is the maximal rank in $(\R^2)^{\otimes 3}$. This example can also be embedded in~$\CC^2 \otimes \CC^3 \otimes \CC^3$ and $\CC^2 \otimes \CC^2 \otimes \CC^3$, where the generic rank is 3 \cite{landsberg2011tensors}. However, note that $\T = \lim_{\epsilon \to 0} \frac{1}{\epsilon} \left((e_1 + \epsilon e_2)^{\otimes 3} - e_1^{\otimes 3}\right)$, so its border rank is two.
\end{Example}

\begin{Example}
    The two-orthogonal tensor 
    \[
    \T = e_1 \otimes e_1 \otimes e_1 \otimes e_1 + e_1 \otimes e_2 \otimes e_1 \otimes e_2 + e_2 \otimes e_1 \otimes e_2 \otimes e_1 + e_2 \otimes e_2 \otimes e_2 \otimes e_2 \in (\R^2)^{\otimes 4}
    \]
    has rank $4$, which is the generic rank in $(\CC^2)^{\otimes 4}$ \cite{gimigliano2011secant}. Looking at the flattenings, we see that the border rank of $\T$ is also $4$. Actually, this decomposition is strong two-orthogonal.
\end{Example}

\begin{Example}
    The two-orthogonal tensor
    \[
    \T = e_1 \otimes e_1 \otimes e_1 + e_1 \otimes e_2 \otimes e_2 + e_1 \otimes e_3 \otimes e_3 + e_2 \otimes e_1 \otimes e_2 + e_3 \otimes e_1 \otimes e_3 \in (\R^3)^{\otimes 3}
    \]
    has rank $5$, which is the generic rank in $(\CC^3)^{\otimes 3}$ \cite{STRASSEN1983645}. The maximal real rank of $3 \times 3 \times 3$ tensors is also $5$ \cite{BREMNER2013401}.
\end{Example}

\begin{Example} \label{ex:rank-proved}
    The two-orthogonal tensor
    \[
    \T = e_1\otimes e_1 \otimes e_1 + e_1 \otimes e_2 \otimes e_2 + e_1 \otimes e_3 \otimes e_3 + e_1 \otimes e_4 \otimes e_4 + e_2 \otimes e_1 \otimes e_2 + e_3 \otimes e_1 \otimes e_3 + e_4 \otimes e_1 \otimes e_4 \in (\R^4)^{\otimes 4}
    \]
    has rank $7$, which is the generic rank in $(\CC^4)^{\otimes 3}$ \cite{LICKTEIG198595}, as follows.
    
    Let $\mathcal{L} = \mathrm{span}\{\T(e_1, \cdot, \cdot), \T(e_2, \cdot, \cdot), \T(e_3, \cdot, \cdot), \T(e_4, \cdot, \cdot)\}$ be the linear space spanned by the slices of $T$ obtained by fixing the first factor.
    The rank of $\T$ is the minimal number of rank-one matrices whose linear span contains $\mathcal{L}$ (e.g. see \cite[Proposition 3.3]{Wang_2023}). The two-orthogonal decomposition above has seven summands, so $\rank(\T) \leq 7$. Suppose that~$\rank(\T) \leq 6$, meaning that there exist six rank-one matrices spanning a space $\sK$ that contains $\sL$. Then, $\sK$ is spanned by $\sL$ along with two rank-one matrices, so every element of~$\sK$ is of the form
    \begin{equation*}\label{eq:span-matrices}
    \begin{pmatrix}
    a_1 & a_2 & a_3 & a_4 \\
    0 & a_1 & 0 & 0 \\
    0 & 0 & a_1 & 0 \\
    0 & 0 & 0 & a_1 \\
    \end{pmatrix} +
    a_5
    \begin{pmatrix}
        x_{11} \\ x_{12} \\ x_{13} \\ x_{14}
    \end{pmatrix} \otimes
    \begin{pmatrix}
        x_{21} \\ x_{22} \\ x_{23} \\ x_{24}
    \end{pmatrix} +
    a_6
    \begin{pmatrix}
        y_{11} \\ y_{12} \\ y_{13} \\ y_{14}
    \end{pmatrix} \otimes
    \begin{pmatrix}
        y_{21} \\ y_{22} \\ y_{23} \\ y_{24}
    \end{pmatrix}
    \end{equation*}
    for fixed $\{x_{ij}\}, \{y_{ij}\}$ and variable coefficients $\{a_i\}$. Computing the $2 \times 2$ minors of these matrices we get that all the rank-one matrices in $\sK$ have $a_1 = 0$, and such matrices do not span $\sL$. Hence, $\rank(\T) = 7$.
\end{Example}

Computing the tensor rank is NP-hard \cite{Has90}. This means that, although a generic tensor has generic rank, it is computationally infeasible to verify whether or not the rank of an arbitrary tensor is smaller, equal, or bigger than the generic rank. We hope that this may be more feasible for two-orthogonal tensors. Perhaps two-orthogonal tensors may be good candidates for tensors with provably high rank, even above the generic rank. Note that the expected generic rank of a tensor of $(\CC^n)^{\otimes d}$ is $\lceil \frac{n^d}{d(n-1) + 1} \rceil$, which is smaller than the maximal two-orthogonal rank in $(\R^n)^{\otimes d}$. 

\subsection{Optimal truncations.}\label{sec:optimal-truncations} When the two-orthogonal rank and the usual rank of a tensor coincide, the two-orthogonal decomposition has good truncation properties. The \emph{$k$-th secant variety} of $X$, denoted $\sigma_k(X)$, is the closure of the set of rank-at-most-$k$ tensors:
\[
\sigma_k(X)=\overline{\bigcup_{x_1,\dots,x_k\in X} \mathrm{span} \{x_1,\dots,x_k\}}.
\]
Consider a rank-$r$ tensor~$\T$ and let $k < r$. A \emph{critical rank-$k$ approximation} of $\T$ is a tensor $\mathcal{S}=x_1+\dots+ x_k\in \sigma_k(X)_{\text{smooth}}$ such that $\T-\mathcal{S}\perp T_{\mathcal{S}} \sigma_k(X)$.

\begin{Proposition}
    Consider a two-orthogonal decomposition $\T = \sum_{i=1}^r x_i \in W_r$. Suppose that $\rank(\T) =r$ and that the truncation $\mathcal{S} = \sum_{i=1}^k x_i$, where $k < r$, is sufficiently general as a point in $\sigma_k(X)$. Then, $\mathcal{S}$ is a critical rank-$k$ approximation of $\T$. 
\end{Proposition}

\begin{proof}
     Suppose that $\mathcal{S}$ is sufficiently general in the sense that it satisfies Terracini's Lemma; i.e. we have $T_{\mathcal{S}}\sigma_k(X) = \mathrm{span}\{ T_{x_1}X,\dots, T_{x_k}X\}$. Recall that the tangent space to $X$ at $x$ is~$T_xX=\sum_{l=1}^d x^{(1)}\otimes\cdots\otimes x^{(l-1)} \otimes \R^{n_l}\otimes x^{(l+1)} \otimes \cdots \otimes x^{(d)}$.
     Two-orthogonality implies $x_j\perp T_{x_i}X$ for all $j>k$ and $i\leq k$. Hence, $\mathcal{S}$ is a critical rank-$k$ approximation of $\T$.
\end{proof}

The previous result suggests that if we order the summands such that $\| x_1\| \geq \cdots \geq \| x_r\|$, then the truncation $\sS = \sum_{i=1}^k x_i$ is a good candidate for a best rank-$k$ approximation of $\T = \sum_{i=1}^r x_i$. Following \cite{vannieuwenhoven2014generic}, an \emph{SEY decomposition} is a tensor decomposition $\T = \sum_{i=1}^r x_i$ such that for all $k < r$, retaining the first $k$ summands gives an optimal solution in the sense that $\sum_{i=1}^k x_i$ is a minimizer of $\|\T - \sS \|$ over all $\sS$ of rank at most $k$.

Two-orthogonality is necessary for having an SEY decomposition \cite[Theorem 3.3]{vannieuwenhoven2014generic}, but it is insufficient because a tensor $\T$ cannot have an SEY decomposition if its border rank and rank disagree (\Cref{ex:rank222}).
Taking flattenings, the singular value decomposition shows that strong two-orthogonality is sufficient to be an SEY decomposition \cite[Theorem 3.6]{vannieuwenhoven2014generic}. The authors in \cite{vannieuwenhoven2014generic} state the following: ``It remains an open question whether strong two-orthogonality is also necessary''. The following counterexample shows that it is not necessary.

\begin{Example}
    Consider the two-orthogonal decomposition
    \begin{equation}\label{eq:SEY}
        \T = 3\,e_1 \otimes e_1 \otimes e_1 + e_2 \otimes e_2 \otimes \left(\frac{1}{2}\,e_1 + 2\,e_2 \right) + e_1 \otimes e_2 \otimes e_3 \in \R^2 \otimes \R^2 \otimes \R^3
    \end{equation}
    We claim that this is an SEY decomposition, despite not being strong two-orthogonal (there is no factor for which the three summands are orthogonal). The border rank of $\T$ is three, e.g. by looking at its flattenings. Critical rank-one approximations correspond to singular vector tuples, by \Cref{prop:tangent-space}. If $x$ is a singular vector tuple of $\T$ with singular value $\lambda$, then $\| \T - \lambda x\|^2 = \| T\|^2 - \lambda^2$. The tensor $\T$ has eight singular vector tuples, which is the generic number for a tensor of this format \cite{FO14}. The singular vector tuple $e_1 \otimes e_1 \otimes e_1$ has singular value $3$, and one can check that all the other singular vector tuples have singular values smaller than $3$. Hence, $3 \, e_1 \otimes e_1 \otimes e_1$ is the best rank-one approximation of $\T$.
    
    The tensor $\T$ has three critical rank-two approximations (which can be computed, for example, using the notion of critical ideal defined in \cite{draisma2016euclidean}). One of them comes from the first two summands in our decomposition: $\sS = 3\,e_1 \otimes e_1 \otimes e_1 + e_2 \otimes e_2 \otimes (1/2\, e_1 + 2\,e_2)$, which leads to a residual $\| \T - \sS\|^2 = 1$. The other two are the tensors
    \[
    \left(\begin{array}{cc||cc||cc}
         \frac{120 - 63\alpha}{38}& 0 & \frac{12 - 12\alpha}{19} & 0 & 0 & 1 \\
         0 & \frac{72 - 53\alpha}{76} & 0 & \alpha & 0 & 0
    \end{array}\right) \in \R^2 \otimes \R^2 \otimes \R^3
    \]
    where $505 \, \alpha^2 - 1010 \, \alpha + 144 = 0$. The rank of these tensors is greater than two, but their border rank is two.
    One can check that both solutions lead to a residual greater than one. Hence, $\sS$ is the best rank-two approximation of $\T$, so \eqref{eq:SEY} is an SEY decomposition.
\end{Example}

\subsection{Two-orthogonality for partially symmetric tensors}
We study two-orthogonal decompositions for partially symmetric tensors. A tensor $\T \in (\R^n)^{\otimes d}$ is symmetric if $\T_{i_1, \dots, i_d} = \T_{\pi(i_1), \dots, \pi(i_d)}$ for any permutation $\pi \in S_d$.
For $d=2$, this says $M = M^\top$. Let~$S^d(\R^n)$ denote the space of symmetric tensors of format $n \times \cdots \times n$ ($d$ times), which is naturally identified with the space of homogeneous polynomials in $n$ variables of degree $d$ \cite{comon2008symmetric}. Using this identification, we express the rank-one tensors in $S^d(\R^n)$ as $\ell^d$ for $\ell \in (\R^n)^\ast$. 

Let $V = S^{d_1} \R^{n_1} \otimes \cdots \otimes S^{d_p} \R^{n_p}$ and let $X  = \{ \ell_1^{d_1} \otimes \cdots \otimes \ell_p^{d_p} \mid \ell_i \in (\R^{n_i})^\ast \} \subset V$ be the cone over the Segre-Veronese variety, the set of rank-one tensors in $V$. An inner product $\langle \cdot, \cdot \rangle$ in $(\R^n)^\ast$ gives a unique inner product on $S^d\R^n$ defined on two rank-one tensors as $\langle f^d, g^d \rangle = \langle f, g\rangle^d$ and extended to all $S^d\R^n$ by linearity. This implies that $\langle f_1 \cdots f_d, g_1 \cdots g_d\rangle = \frac{1}{d!}\sum_{\pi \in S_d} \prod_{k=1}^d \langle f_k, g_{\pi(k)}\rangle$. As before, a set of inner products in $(\R^{n_1})^\ast, \dots, (\R^{n_p})^\ast$ defines an inner product in $V$.

\begin{Definition}
    A partially symmetric tensor $\T \in  S^{d_1} \R^{n_1} \otimes \cdots \otimes S^{d_p} \R^{n_p}$ is called \emph{two-orthogonal} if it admits a decomposition $\T = \sum_{i=1}^r x_i$ where $x_i = \ell_{1,i}^{d_1} \otimes \cdots \otimes \ell_{p,i}^{d_p} \neq 0$ and all pairs of summands $x_i$ and $x_j$ with $i \neq j$ are orthogonal in at least two factors, counting multiplicities $d_k$. We define $W_r$ as the closure of the set of length-at-most-$r$ two-orthogonal partially symmetric tensors, and $W = \overline{\bigcup_r W_r}$ as the \emph{two-orthogonal variety}.
\end{Definition}

\begin{Example}
    The tensor $f_1^2 \otimes g_1 + f_2^2 \otimes g_2 \in S^2\R^m \otimes \R^n$ with $f_1 \perp f_2$ is two-orthogonal.
\end{Example}
\begin{Proposition} \label{prop:tangent-partially-symmetric}
Let $x = \ell_1^{d_1} \otimes \cdots \otimes \ell_p^{d_p} \in X\setminus\{ 0\} \subseteq V$.
\begin{enumerate}
    \item The tangent space to $X$ at $x$ is $T_xX = \left\{ \sum\limits_{k=1}^d \ell_{1}^{d_1} \otimes \cdots \otimes \ell_k^{d_k - 1} \ell \otimes \cdots \otimes \ell_p^{d_p} \mid \ell \in (\R^{n_k})^\ast \right\}$.
\item If $d_k \geq 2$ for all $k$, then $N_xX \cap X = \left\{ y_1^{d_1} \otimes \cdots \otimes y_p^{d_p} \mid  \langle \ell_k, y_k\rangle = 0, \text{ for some $k$}\right\}$.
\end{enumerate}
\end{Proposition}
\begin{proof}
    Part (1) follows by applying the Leibniz rule to the parametrization of $X$. For part~(2), suppose that $y=y_1^{d_1} \otimes \cdots \otimes y_1^{d_1}\in N_xX$, then 
    \[
    \langle y , \ell_{1}^{d_1} \otimes \cdots \otimes \ell_k^{d_k - 1} \ell \otimes \cdots \otimes \ell_p^{d_p}\rangle=\langle y_k,\ell_k\rangle^{d-1}\langle y_k,\ell\rangle\prod_{j\neq k}\langle y_j,\ell_j\rangle^{d_j}=  0 
    \]
    for all $k\in [p]$ and all $ \ell \in (\R^{n_k})^\ast$. So there exists $k\in[p]$ such that $\langle y_k,\ell_k\rangle=0$.
\end{proof}

\Cref{prop:tangent-partially-symmetric} implies that \Cref{thm:non_order_dep} also holds for partially symmetric tensors. \Cref{thm:maximal-partially-symmetric} is an extension of \Cref{thm:maximal_length} to partially symmetric tensors. The proof is similar; we include it for completeness.
\begin{Theorem} \label{thm:maximal-partially-symmetric}
     Let $V = S^{d_1} \R^{m_1} \otimes \cdots \otimes S^{d_p} \R^{m_p} \otimes \R^{n_1} \otimes \cdots \otimes \R^{n_q}$, $n_1 \leq \cdots \leq n_q$ and $d_k \geq 2$ for all $k \in [p]$. Then, the maximal length of a two-orthogonal decomposition is $N = m_1 \cdots m_p n_1 \cdots n_{q-1}$. In particular, $W=W_N$.
\end{Theorem}
    
\begin{proof}
    Let $\T = \sum_{i=1}^r \ell_{1,i}^{d_1} \otimes \cdots \otimes \ell_{p,i}^{d_p} \otimes u_{1,i} \otimes \cdots \otimes u_{q,i} \in V$ be two-orthogonal. Consider the tensor $\T' = \sum_{i=1}^r \ell_{1,i} \otimes \cdots \otimes \ell_{p,i} \otimes u_{1,i} \otimes \cdots \otimes u_{{q-1},i}$. Each pair of summands of $\T'$ are orthogonal in at least one factor, so they are linearly independent. Therefore, $r \leq m_1 \cdots m_p n_1 \cdots n_{q-1}$. It remains to show that the bound can be achieved: Let $\{e_{k,1},\dots, e_{k,m_k}\}$ be an orthogonal basis of $(\R^{m_k})^\ast$,  let $\{e'_{k,1},\dots, e'_{k,n_k}\}$ be an orthogonal basis of $\R^{n_k}$, and let $\sJ \subseteq [n_1] \times \cdots \times [n_q]$ be a maximal set of indices given by a Latin hypercube, as in \Cref{lemma:construction_maximal}. Then, the tensor
    \[
    \T = \sum_{\bj \in \sJ, i_1, \cdots, i_p} e_{1,i_1}^{d_1} \otimes \cdots \otimes e_{p,i_p}^{d_p} \otimes e'_{1,j_1} \otimes \cdots \otimes e'_{d,j_d}
    \]
    is two-orthogonal and consists of $m_1 \cdots m_p n_1 \cdots n_{q-1}$ terms.
\end{proof}

\section{Dimension of the two-orthogonal variety} \label{sec:two-orthogonal-variety}

In this section, we prove \Cref{thm:dimension_lower_bound}, which lower bounds the dimension of the two-orthogonal variety. To do this, we construct a dimension-preserving map that parametrizes a set of basis-aligned two-orthogonal tensors. Recall that the \emph{Hamming distance} on $[n]^d$ is~$d_H(\bi, \bj) = |\{ k \mid i_k \neq j_k\}|$.

\begin{Lemma}\label{lemma:distance3}
Let $n\geq 2, d\geq 3$, and let $L$ be the Latin hypercube on $[n]^{d-1}$ cells given by~$L(i_1, \dots, i_{d-1}) = i_{d-1} + \sum_{k =1}^{d-2}(-1)^{k+1} i_k \mod n$. Let $\sI = \{(i_1, \dots, i_{d-1}, L(i_1, \dots, i_{d-1})) \}$. Then there exists a subset $\sJ \subset \sI$ with $|\sJ| = n-1$ and $d_H(\textnormal{\bi}, \textnormal{\bj}) \geq 3$ for all distinct $\textnormal{\bi}, \textnormal{\bj} \in \sJ$.
\end{Lemma}
\begin{proof}
    If $d$ is even, then $L(i,\dots, i) = i$, so we can let $\sJ = \{ (i, i, \dots, i) \mid i \in [n-1]\}$. If $d$ is odd, then $L(i,\dots, i) = 2i \mod n$, so we consider two cases. If $n$ is odd, let $\sJ = \{(i, \dots, i, 2i\mod n)  \mid i \in [n-1] \}$. If $n$ is even, let 
    \[
    \sJ = \left\{ (i, \dots, i, i, 2i) \mid 1 \leq i \leq \frac{n}{2} \right\} \cup \left\{ (i, \dots, i, i+1, 2i+1 - n) \mid \frac{n}{2} < i \leq n-1 \right\}. \qedhere
    \]
\end{proof}

\dimensionlowerbound*
\begin{proof} We parameterize a set of two-orthogonal tensors via a dimension-preserving map. This lower bounds the dimension.
Let $L$ be the Latin hypercube in \Cref{lemma:distance3}, and let $\sI = \{(i_1, \dots, i_{d-1}, L(i_1, \dots , i_{d-1})) \mid i_j \in [n] \}$. Given a set of scalars $\{ \lambda_{\bi} \mid \bi \in \sI \} \subset \R$ and an orthonormal basis $\{u^{(k)}_1, \dots, u^{(k)}_n \} \subset \R^n$ for each $k \in [d]$, we construct
\[
\T = \sum_{\bi \in \sI} \lambda_\bi u^{(1)}_{i_1} \otimes \cdots \otimes u^{(d)}_{i_d} \in W_{n^{d-1}}.
\]
There are $n^{d-1}$ degrees of freedom to choose the $\lambda_\bi$'s and, for each factor $k \in [d]$, we have~$\binom{n}{2}$ degrees of freedom to choose $\{u^{(k)}_1, \dots, u^{(k)}_n \}$. Hence, the expected dimension of the variety of tensors with such a decomposition is $n^{d-1} + d \binom{n}{2}$.

We can rescale each $u_i^{(k)}$ and scale the $\lambda_\bi$'s accordingly. Hence, in the following parametrization we fix an entry of the vectors to be one, instead of imposing that they have norm one. Let~$\theta = (\theta_{ij} \mid 1 \leq j < i \leq n) \in \R^{\binom{n}{2}}$ and consider the matrix
\begin{equation*} \label{eq:u-orthogonal}
    U(\theta) = \begin{pmatrix}
    | & & | \\ 
    u_1(\theta) & \cdots & u_{n}(\theta) \\
    | & & | \\ 
\end{pmatrix} = 
\begin{pmatrix}
    1 & p_{1,2}(\theta) & p_{1,3}(\theta) & \cdots & p_{1,n}(\theta) \\
    \theta_{2,1} & 1 & p_{2,3}(\theta) & \cdots & p_{2,n}(\theta) \\
    \theta_{3,1} & \theta_{3,2} & 1 & \cdots & p_{3,n}(\theta) \\
    \vdots & \vdots &\vdots & \ddots & \vdots\\
    \theta_{n,1} & \theta_{n,2} & \cdots & \cdots & 1 \\
\end{pmatrix}
\end{equation*}
where $p_{ij}(\theta) \in \R(\theta)$ are chosen such that the columns of $U$ are orthogonal (for simplicity in notation, we consider the Euclidean inner product throughout this proof). 
We parametrize a family of basis-aligned tensors in $W_{n^{d-1}}$ by
\[
\begin{matrix}
    \phi : \; & \R^{n^{d-1}} &\times& \left(\R^{\binom{n}{2}} \right)^d & \to & \left(\R^n\right)^{\otimes d} \\
    & (\lambda_\bi)&,& \theta^{(1)}, \dots, \theta^{(d)} & \mapsto & \sum_{\bi \in \sI} \lambda_\bi u_{i_1}(\theta^{(1)}) \otimes \cdots \otimes u_{i_d}(\theta^{(d)}).
\end{matrix}
\]
We show that the Jacobian of $\phi$ is generically full-rank. In other words, given a generic $\sS \in \text{Im}(\phi) \subseteq W_{n^{d-1}}$, we show that $J\phi : \R^{n^{d-1}} \times \left(\R^{\binom{n}{2}} \right)^d  \to  T_\sS W_{n^{d-1}}$ is injective. We study $\left.J\phi \right\rvert_{\theta^{(1)} = \cdots = \theta^{(d)} = 0}$. Orthogonality implies that $p_{ij}(0) = 0$ for all $i,j$. Let $i<j$, then
\[
     0 = \langle u_i(\theta), u_j(\theta) \rangle = \sum_{k < i } p_{ki}(\theta) p_{kj}(\theta) + p_{ij}(\theta) + \sum_{i < k < j} \theta_{ki} p_{kj}(\theta) + \theta_{ji} + \sum_{k > j} \theta_{ki}\theta_{kj}.
    \]
Taking derivatives and evaluating at $\theta = 0$, we get
\[
\left.\frac{\partial p_{ij}(\theta)}{\partial \theta_{kl}} \right\rvert_{\theta = 0} = \begin{cases}
    -1 \quad &\text{if} \quad (i,j) = (l,k) \\
    0 \quad &\text{otherwise}.
\end{cases}
\]
In what follows, we write $\Theta = 0$ to denote $\theta^{(1)} = \cdots = \theta^{(d)} = 0$. Note that $d_H(\bi,\bj) \geq 2$ for all distinct $\bi,\bj \in \sI$. Hence
\begin{align*}
    \left.\frac{\partial \phi_{\bi}}{\partial \lambda_{\bj}} \right\rvert_{\Theta = 0} &= \begin{cases}
        \delta_{\bi \bj} & \text{if} \quad \bj \in \sI \\
        0 & \text{otherwise,}
    \end{cases}\\  
    \left.\frac{\partial \phi_{\bi}}{\partial \theta^{(k)}_{\ell_1 \ell_2}} \right\rvert_{\Theta = 0} &= \begin{cases}
        \pm \lambda_\bj &\text{if $\bj \in \sI$, $d_H(\bi, \bj) = 1$, $i_k \neq j_k$, and $\{\ell_1, \ell_2 \} = \{ i_k, j_k\}$.} \\
        0 & \text{otherwise.}
    \end{cases}
\end{align*}
The sign in the last expression depends on $L$ and whether~$i_k < j_k$. Note that $\left.\frac{\partial \phi_{\bi}}{\partial \theta^{(k)}_{\ell_1 \ell_2}} \right\rvert_{\Theta = 0} = 0$, if $\bi \in \sI$ and $\left.\frac{\partial \phi_{\bi}}{\partial \lambda_{\bj}} \right\rvert_{\Theta = 0}= 0$, if $\bi \notin \sI$. Therefore, after reordering rows, $\left.J\phi \right\rvert_{\Theta = 0}$ is a block matrix. The block of partial derivatives with respect to $\{\lambda_{\bi}\}$ is a full-rank $n^{d-1}\times n^{d-1}$ matrix (a permutation matrix). The other block, corresponding to partial derivatives with respect to~$\{\theta^{(k)}_{\ell_1 \ell_2}\}$, is an $ n^{d-1}(n-1) \times d\binom{n}{2}$ matrix. Entries outside these two blocks are zero (see \Cref{ex:jacobian}). We show that $\left.J\phi \right\rvert_{\Theta = 0}$ is full-rank for generic $\lambda_\bi$'s.
Let $\sJ$ be as in \Cref{lemma:distance3} and consider the following sub-matrix of $\left.J\phi \right\rvert_{\Theta = 0}$ where we set $\lambda_{\bi}$ to zero if $\bi \notin \sJ$:
\[
    M = \left( \left.\frac{\partial \phi_{\bi}}{\partial \theta^{(k)}_{\ell_1 \ell_2}} \right\rvert_{\Theta = 0, \; \lambda_\bi = 0 \text{ if } \bi \notin \sJ}\right)_{\bi, (k, \ell_1, \ell_2)} \in \R^{n^d \times d \binom{n}{2}}.
\]
Since $|\sJ| = (n-1)$, for every $k \in [d]$ and every $ \ell_2 < \ell_1 \in [n]$, there is some $\bj \in \sJ$ such that~$j_k = \ell_1$ or $j_k = \ell_2$, so every column of $M$ is nonzero. Moreover, we have $d_H(\bj_1, \bj_2) \geq 3$ for all distinct $\bj_1, \bj_2 \in \sJ$, so there is no $\bi \in [n]^d$ with $d_H(\bi, \bj_1) = d_H(\bi, \bj_2) = 1$. Hence, every row of $M$ has at most one nonzero entry. Therefore, $M$ is full-rank as long as $\lambda_\bi \neq 0$ for~$\bi \in \sJ$. By continuity of the determinant, $\left.J\phi \right\rvert_{\Theta = 0}$ is full-rank for generic $\lambda_\bi$'s.
\end{proof}

\begin{Example}\label{ex:jacobian}
    The Jacobian of $\phi$ in the proof of \Cref{thm:dimension_lower_bound} is given below for $(\R^2)^{\otimes 3}$:
\[
\left.J\phi \right\rvert_{\Theta = 0} = 
\begin{blockarray}{cccccccccc}
& & \lambda_{112} & \lambda_{121} & \lambda_{211} & \lambda_{222} & \theta_{21}^{(1)} & \theta_{21}^{(2)} & \theta_{21}^{(3)} &\\[1em]
\begin{block}{cc(ccccccc)c}
  112 & & 1 &  &  &  &  & & &\\
  121 & &  & 1 &  &  &  & & &\\
  211 & &  &  & 1 &  &  & & &\\
  222 & &  &  &  & 1 &  & & &.\\
  111 & &  &  &  &  & -\lambda_{211} & -\lambda_{121} & -\lambda_{112} &\\
  122 & &  &  &  &  & -\lambda_{222} & \lambda_{112} & \lambda_{121} &\\
  212 & &  &  &  &  & \lambda_{112} & -\lambda_{222} & \lambda_{211} &\\
  221 & &  &  &  &  & \lambda_{121} & \lambda_{211} & -\lambda_{222}&\\
\end{block}
\end{blockarray}
 \]
 This matrix is full-rank for generic $\lambda_\bi \in \R$. Hence, the two-orthogonal decompositions
 \[
 \lambda_{112} u_1^{(1)} \otimes u_1^{(2)} \otimes u_2^{(3)} + \lambda_{121} u_1^{(1)} \otimes u_2^{(2)} \otimes u_1^{(3)} + \lambda_{211} u_2^{(1)} \otimes u_1^{(2)} \otimes u_1^{(3)} + \lambda_{222} u_2^{(1)} \otimes u_2^{(2)} \otimes u_2^{(3)},
 \]
 where $u_1^{(k)} = (1, \theta_{21}^{(k)}) \perp (-\theta_{21}^{(k)}, 1) = u_2^{(k)}$, parametrize a variety of dimension $4 + 3 \binom{2}{2} = 7$.
\end{Example}

We have seen how some components defined by maximal basis-aligned tensors are nondefective, i.e. the dimension is the expected one. However, these components could be part of a higher-dimensional component. Indeed, there exist two-orthogonal decompositions of maximal length that are not basis-aligned, which implies having more degrees of freedom in the parametrization. The following example shows that the lower bound of \Cref{thm:dimension_lower_bound} is not tight. Calculating the dimension of the two-orthogonal variety remains an open question.

\begin{Example}\label{ex:non-basis-aligned444} Consider the following family of two-orthogonal tensors in $(\R^4)^{\otimes 3}$
    \begin{align*}
    & \lambda_{111}\,e_1 \otimes e_1 \otimes e_1 + \lambda_{122}\,e_1 \otimes e_2 \otimes e_2 + \lambda_{133}\,e_1' \otimes e_3 \otimes e_3 + \lambda_{144}\,e_1' \otimes e_4 \otimes e_4 \\
    + & \lambda_{212}\,e_2 \otimes e_1 \otimes e_2 + \lambda_{221}\,e_2 \otimes e_2 \otimes e_1 + \lambda_{234}\,e_2' \otimes e_3 \otimes e_4 + \lambda_{243}\,e_2' \otimes e_4 \otimes e_3 \\
    + & \lambda_{313}\,e_3 \otimes e_1' \otimes e_3' + \lambda_{324}\,e_3 \otimes e_2' \otimes e_4' + \lambda_{331}\,e_3' \otimes e_3' \otimes e_1' + \lambda_{342}\,e_3' \otimes e_4' \otimes e_2' \\
    + & \lambda_{414}\,e_4 \otimes e_1' \otimes e_4' +  \lambda_{423}\,e_4 \otimes e_2' \otimes e_3' + \lambda_{432}\,e_4' \otimes e_3' \otimes e_2' + \lambda_{441}\,e_4' \otimes e_4' \otimes e_1' 
\end{align*}
where $ e_1', e_2' \in \mathrm{span}\{e_1, e_2\}$, $e_3', e_4' \in \mathrm{span}\{e_3, e_4\}$, $\langle e_1', e_2'\rangle = \langle e_3', e_4'\rangle = 0$, and $\|e_i'\|=1$ for all~$i$. We have $\binom{4}{2} = 6$ degrees of freedom to choose $\{e_1, \dots, e_4\}$, one degree of freedom to choose $\{e_1', e_2'\}$, and one degree of freedom to choose $\{e_3', e_4'\}$. This adds up to eight degrees of freedom for each factor, so a total of 24. We can arbitrarily scale each summand, which adds 16 degrees of freedom. Using the computer algebra software \texttt{Macaulay2} \cite{M2}, we computed the Jacobian of the corresponding parametrization and checked that it is generically full-rank. Hence, the dimension of the variety parametrized by these decompositions is 40, while the lower bound given by Theorem $\ref{thm:dimension_lower_bound}$ is 34.
\end{Example}

\section{Parametric description of the two-orthogonal variety}\label{sec:combinatorial_descriptions}

We give a combinatorial parametrization of the two-orthogonal variety using a graph-theoretical perspective. For binary tensors, the resulting graphs are bipartite; we study this case in more depth.

\subsection{Graphical descriptions.} A \emph{graph} is a pair $\sG = (V(\sG), E(\sG))$ where $V(\sG)$ is a set of vertices and $E(\sG)$ is a set of unordered pair of vertices, called edges. We describe two-orthogonal decompositions via graphs, as follows.

\begin{Definition}
    Let $V = \R^{n_1} \otimes \cdots \otimes \R^{n_d}$. The \emph{graphical description} of a two-orthogonal decomposition $\T = \sum_{i=1}^r x_i^{(1)} \otimes \cdots \otimes x_i^{(d)} \in V$ is a tuple of $d$ graphs $(\sG_1, \dots, \sG_d)$ such that~$V(\sG_k) = [r]$ and $\{ i,j\} \in E(\sG_k)$ if and only if $x_i^{(k)} \perp x_j^{(k)}$.
\end{Definition}

\begin{Example}
    The graphical description of an odeco decomposition of rank $r$ is $(K_r, \dots, K_r)$, where~$K_r$ is the complete graph on $r$ vertices.
\end{Example}

\begin{Example} \label{ex:graphs}
Consider the tensor $e_1 \otimes e_1 \otimes e_1 + e_2 \otimes (e_1 + e_2) \otimes e_2 + (e_1 - e_2) \otimes e_3 \otimes e_3$. Of the three vectors $x_i^{(1)}$ appearing in the first factor, the first and second vectors are orthogonal. In the second factor, the first and third vectors are orthogonal and the second and third vectors are orthogonal. In the third factor, all three vectors are orthogonal. Hence the graphical description of the decomposition is the following.
    \begin{figure}[ht]
    \vspace{-0.1cm}
    \captionsetup[subfigure]{labelformat=empty}
    \centering
    \begin{subfigure}{0.2\textwidth}
        \centering
        \begin{tikzpicture}
            \node[draw,circle,minimum size=0.5cm,inner sep=0] (n1) at (0,0) {1};
            \node[draw,circle,minimum size=0.5cm,inner sep=0] (n2) at (1,0) {2};
            \node[draw,circle,minimum size=0.5cm,inner sep=0] (n3) at (0.5,0.86) {3};
            \draw (n1) -- (n2); 
        \end{tikzpicture}
        \caption{$\sG_1$}
    \end{subfigure}
    \hspace{0.1cm}
    \begin{subfigure}{0.2\textwidth}
        \centering
        \begin{tikzpicture}
            \node[draw,circle,minimum size=0.5cm,inner sep=0] (n1) at (0,0) {1};
            \node[draw,circle,minimum size=0.5cm,inner sep=0] (n2) at (1,0) {2};
            \node[draw,circle,minimum size=0.5cm,inner sep=0] (n3) at (0.5,0.86) {3};
            \draw (n1) -- (n3);
            \draw (n2) -- (n3);
        \end{tikzpicture}
        \caption{$\sG_2$}
    \end{subfigure}
    \hspace{0.1cm}
    \begin{subfigure}{0.2\textwidth}
        \centering
        \begin{tikzpicture}
            \node[draw,circle,minimum size=0.5cm,inner sep=0] (n1) at (0,0) {1};
            \node[draw,circle,minimum size=0.5cm,inner sep=0] (n2) at (1,0) {2};
            \node[draw,circle,minimum size=0.5cm,inner sep=0] (n3) at (0.5,0.86) {3};
            \draw (n1) -- (n2);
            \draw (n1) -- (n3);
            \draw (n2) -- (n3);
        \end{tikzpicture}
        \caption{$\sG_3$}
    \end{subfigure}
    \vspace{-0.4cm}
\end{figure}  
\end{Example}

A \emph{multigraph} is $\sG = (V(\sG), E(\sG))$ where $V(\sG)$ is the set of vertices and $E(\sG)$ is a \emph{multiset} of unordered pairs of vertices. Given $\mu, n \in \NN$, the complete multigraph $\mu K_n$ is the complete graph $K_n$ in which every edge appears $\mu$ times. We define the union of two graphs $\sG_1, \sG_2$ as the multigraph $\sG = \sG_1 \cup \sG_2$, where $V(\sG) = V(\sG_1) \cup V(\sG_2)$ and $E(\sG) = E(\sG_1) \sqcup E(\sG_2)$.
Given a pair of (multi)graphs $\sG_1, \sG_2$, we say that $\sG_1 \subseteq \sG_2$ if $V(\sG_1) \subseteq V(\sG_2)$ and $E(\sG_1) \subseteq E(G_2)$. For example, the multigraph obtained as the union of the three graphs from \Cref{ex:graphs} is
\begin{figure}[ht]
    \vspace{-0.1cm}
    \captionsetup{labelformat=empty}

        \centering
        \begin{tikzpicture}
            \node[draw,circle,minimum size=0.5cm,inner sep=0] (n1) at (0,0) {1};
            \node[draw,circle,minimum size=0.5cm,inner sep=0] (n2) at (1,0) {2};
            \node[draw,circle,minimum size=0.5cm,inner sep=0] (n3) at (0.5,0.86) {3};
            \draw (n1) to[bend left=15] (n2); 
            \draw (n1) to[bend right=15] (n2);
        
            \draw (n1) to[bend left=15] (n3); 
            \draw (n1) to[bend right=15] (n3);
        
            \draw (n3) to[bend left=15] (n2); 
            \draw (n3) to[bend right=15] (n2);
        \end{tikzpicture}
        \vspace{-0.3cm}
        \caption{$\sG = \sG_1 \cup \sG_2 \cup \sG_3$}
    \vspace{-0.4cm}
\end{figure}  

\begin{Definition}[\cite{haynes2010orthogonal}]
    Given a (multi)graph $\sG$ and a positive integer $n$, an \emph{orthogonal vector $n$-coloring} of $\sG$ is an assignment of vectors of $\R^n \setminus \{ 0\}$ to $V(\sG)$ such that if $\{v,w\} \in E(\sG)$ then~$v$ and $w$ receive orthogonal vectors. The \emph{vector chromatic number} $\chi_v(\sG)$ is the least $n$ such that $\sG$ has an orthogonal vector $n$-coloring.
\end{Definition}
\begin{Proposition} \label{prop:graphical-description}
    Let $(\sG_1, \dots, \sG_d)$ be the graphical description of a two-orthogonal decomposition $\T = \sum_{i=1}^r x_i^{(1)}\otimes \cdots \otimes x_i^{(d)} \in \R^{n_1} \otimes \cdots \otimes \R^{n_d}$. Then $\chi_v(\sG_j) \leq n_j$ for all $j \in [d]$, and~$2K_r \subseteq  \bigcup_{j} \sG_j$.
\end{Proposition}
\begin{proof}
    For each $j \in [d]$, assigning $x_i^{(j)}$ to vertex $i \in [r]$ gives an orthogonal vector $n_j$-coloring of $\sG_j$. The fact that $2K_r \subseteq  \bigcup_{j} \sG_j$ follows from two-orthogonality.
\end{proof}

In view of \Cref{prop:graphical-description}, we define the notion of a valid graphical description for the two-orthogonal variety $W \subset \R^{n_1} \otimes \cdots \otimes \R^{n_d}$ to mean a tuple of $d$ graphs that encode the orthogonality relations of some two-orthogonal tensor in $\R^{n_1} \otimes \cdots \otimes \R^{n_d}$.

\begin{Definition}
     A \emph{valid graphical description} for $W \subset \R^{n_1} \otimes \cdots \otimes \R^{n_d}$ is a tuple of graphs $(\sG_1, \dots, \sG_d)$ such that $V(\sG_1) = \cdots = V(\sG_d) =[r]$, $\chi_v(\sG_j) \leq n_j$ for all $j$, and $2K_r \subseteq  \bigcup_{j} \sG_j$.
\end{Definition}

\begin{Remark}
    \Cref{lemma:stabilization} can be proved using the combinatorial tools in this section. Given two graphs $\sG, \sG'$ on the same vertex set, we have $\chi_v(\sG \cup \sG') \leq \chi_v(\sG) \chi_v(\sG')$ \cite[Lemma 1.14]{haynes2010orthogonal}. Fix $n_1 \leq \cdots \leq n_d$ and let $(\sG_1, \dots, \sG_d)$  with $V(\sG_j) = [r]$ for all $j$ be a valid graphical description for $W \subset \R^{n_1} \otimes \cdots \otimes \R^{n_d}$. Since $2K_r \subseteq \bigcup_{j=1}^{d} \sG_j$ we have $K_r \subseteq \bigcup_{j=1}^{d-1} \sG_j$, so
    \[
    r = \chi_v(K_r) = \chi_v\left(\bigcup_{j=1}^{d-1} \sG_j\right) \leq \chi_v(\sG_1) \cdots \chi_v(\sG_{d-1}) \leq n_1 \cdots n_{d-1}.
    \]
\end{Remark}

Given a valid graphical description for $W \subset \R^{n_1} \otimes \cdots \otimes \R^{n_d}$, we are interested in obtaining the two-orthogonal decompositions that have that graphical description. 

\begin{Definition}
    The \emph{orthogonal vector $n$-coloring ideal} of a graph $\sG$ is
    \[
    I_{\sG, n} = \left( \langle u_i, u_j \rangle \mid \{i,j\} \in E(\sG) \right) + \left( \langle u_i, u_i \rangle - 1 \mid i \in V(\sG) \right) 
    \subset \R[u_{i,1}, \dots, u_{i,n} \mid i \in V(\sG)].
    \]
\end{Definition}
The zero locus $Z(I_{\sG, n}) \subseteq (\R^{n})^{\times |V(\sG)|}$ consists of all orthogonal vector $n$-colorings of~$\sG$ with unit vectors.
If $\chi_v(\sG) \leq n$, then~$Z(I_{\sG, n}) \neq \varnothing$ by definition. This implies the following.

\begin{Proposition}
    The two-orthogonal variety $W \subset \R^{n_1} \otimes \cdots \otimes \R^{n_d}$ is equal to the union of~$\overline{\mathrm{Im}(\psi_{(\sG_1, \dots, \sG_d)})}$ over all valid graphical descriptions $(\sG_1, \dots, \sG_d)$ for $W$, where
    \[
    \begin{array}{r@{\hspace{0cm}}c@{\hspace{0cm}}c@{\hspace{0.1cm}}c@{\hspace{0.1cm}}c@{\hspace{0cm}}c@{\hspace{0cm}}c@{\hspace{0cm}}ccc}
    \psi_{(\sG_1, \dots, \sG_d)} :& Z(I_{\sG_1, n_1}) &\times &\cdots &\times& Z(I_{\sG_d, n_d}) & \times & \R^r& \to & \R^{n_1} \otimes \cdots \otimes \R^{n_d} \\
          &\big((x_1^{(1)}, \dots, x_r^{(1)})&, & \dots & ,&(x_1^{(d)}, \dots, x_r^{(d)}) &, &(\lambda_1, \dots, \lambda_r) \big)& \mapsto & \sum\limits_{i=1}^r \lambda_i x_i^{(1)} \otimes \cdots \otimes x_i^{(d)}.
    \end{array}
    \]
\end{Proposition}

\begin{Remark} \label{remark:not-prime}
We cannot guarantee that $\overline{\mathrm{Im}(\psi_{(\sG_1, \dots, \sG_d)})}$ is irreducible, since $I_{\sG, n}$ is not prime in general. For example, for $\sG = ([4], \{\{1,2\}, \{1,3\}, \{1,4\}, \{2,3\},\{2,4\}\})$, $I_{\sG, 3}$ is generated by
\[
\begin{array}{ccc}
u_{1, 1}u_{2, 1}+u_{1, 2}u_{2, 2}+u_{1, 3}u_{2, 3}, & u_{1, 1}u_{3, 1}+u_{1, 2}u_{3, 2}+u_{1, 3}u_{3, 3}, & u_{1, 1}u_{4, 1}+u_{1, 2}u_{4, 2}+u_{1, 3}u_{4, 3},\\ u_{2, 1}u_{3, 1}+u_{2, 2}u_{3, 2}+u_{2, 3}u_{3, 3}, & u_{2, 1}u_{4, 1}+u_{2, 2}u_{4, 2}+u_{2, 3}u_{4, 3}, & u_{1, 1}^{2}+u_{1, 2}^{2}+u_{1, 3}^{2}-1,\\ u_{2, 1}^{2}+u_{2, 2}^{2}+u_{2, 3}^{2}-1, & u_{3, 1}^{2}+u_{3, 2}^{2}+u_{3, 3}^{2}-1, & u_{4, 1}^{2}+u_{4, 2}^{2}+u_{4, 3}^{2}-1,
\end{array}
\]
and its primary decomposition consists of four components.
\end{Remark}

\begin{Remark} \label{remark:inclusion_components}
    Graphical descriptions give inclusions between components of $W$. We define a partial order on the valid graphical descriptions for $W$ as follows: 
    $(\sG_1, \dots, \sG_d) \preceq (\sG_1', \dots, \sG_d')$ if $V(\sG_j) = V(\sG_j')$ and $E(\sG_j') \subseteq E(\sG_j)$ for all $j \in [d]$. If $(\sG_1, \dots, \sG_d) \preceq (\sG_1', \dots, \sG_d')$, then $\overline{\mathrm{Im}(\psi_{(\sG_1, \dots, \sG_d)})} \subseteq \overline{\mathrm{Im}(\psi_{(\sG_1', \dots, \sG_d')})}$.
    The graphical description depends on the order of the summands. To avoid this dependence, the partial order can be defined up to graph isomorphism. This leads to a notion of maximal valid graphical descriptions. Counting and characterizing the maximal valid graphical descriptions remains an open question.
\end{Remark}

\subsection{Binary tensors} \label{sec:binary}

In this section, we focus on two-orthogonal tensors in $V = (\R^2)^{\otimes d}$. Let $(\sG_1, \dots, \sG_d)$ be the graphical description of a two-orthogonal decomposition in $V$. Then, each $\sG_j$ is a disjoint union of complete bipartite graphs, as follows. If $u, v, w \in \R^2$, $u \perp v$, and $u \perp w$, then $v$ and $w$ are collinear. Therefore, without loss of generality, in this section we will only consider valid graphical descriptions such that each graph is a disjoint union of complete bipartite graphs.

\Cref{remark:not-prime} shows that $I_{\sG,n}$ is not prime in general, so a valid graphical description containing $\sG$ may not parametrize an irreducible variety. The following result shows that these varieties are irreducible for binary tensors.

\begin{Proposition}\label{prop:expected-dim}
    Fix a tuple of graphs $(\sG_1, \dots, \sG_d)$ on vertex set $[r]$. Assume that it is a valid graphical description for $W \subset (\R^2)^{\otimes d}$, so each $\sG_j$ is a disjoint union of $c_j$ complete bipartite graphs. Then $\overline{\mathrm{Im}(\psi_{(\sG_1, \dots, \sG_d)})}$ is irreducible and its expected dimension is $ r + \sum_{j=1}^d c_j$.
\end{Proposition}
\begin{proof}
    First, suppose that $\sG$ is a complete bipartite graph with $V(\sG) = [r]$. Let $(e_{i_1}, \dots, e_{i_r})$ with $i_j \in \{1,2\}$ be an orthogonal vector two-coloring of $\sG$. Acting with $\SO(2)$, we get all orthogonal vector two-colorings of $\sG$ with unit vectors, up to sign. This is an irreducible subvariety of $Z (I_{\sG,2})$ of dimension one, let us call it $Y_\sG$. Now suppose that $\sG$ is a disjoint union of complete bipartite graphs, i.e. $\sG = \sH_1 \sqcup \cdots \sqcup \sH_c$. Then, $Y_\sG = Y_{\sH_1} \times \cdots \times Y_{\sH_c}$ is a product of irreducible varieties, hence irreducible, and its dimension is $c$.
    
    Using the previous reasoning for each $\sG_j$, we get that $Y_{\sG_1} \times \cdots \times Y_{\sG_d} \times \R^r$ is a product or irreducible varieties, hence irreducible, and its dimension is $r + \sum_{j=1}^d c_j$. In the tensor decomposition given by $\psi_{(\sG_1, \dots, \sG_d)}$, the signs of the vectors can be absorbed by $\lambda \in \R^r$, so 
    \[
    \psi_{(\sG_1, \dots, \sG_d)} (Z(I_{\sG_1}) \times \cdots \times Z(I_{\sG_d}) \times \R^r) = \psi_{(\sG_1, \dots, \sG_d)}(Y_{\sG_1} \times \cdots \times Y_{\sG_d} \times \R^r).
    \]
    Finally, since $\psi_{(\sG_1, \cdots, \sG_d)}$ is a polynomial map, the variety $\overline{\mathrm{Im}(\psi_{(\sG_1, \dots, \sG_d)})}$ is also irreducible, and its expected dimension is $\dim(Y_{\sG_1} \times \cdots \times Y_{\sG_d} \times \R^r) = r + \sum_{j=1}^d c_j$.
\end{proof}

\Cref{thm:dimension_lower_bound} says that $\dim(W) \geq 2^{d-1} + d$. We conjecture that for binary tensors we have equality, and the top-dimensional component is $\overline{W_{2^{d-1}} \setminus W_{2^{d-1} - 1}}$.

\begin{Conjecture} \label{conj:maximal_dimension_binary}
    Fix a tuple of graphs $(\sG_1, \dots, \sG_d)$ on vertex set $[r]$. Assume that it is a valid graphical description for $W \subset (\R^2)^{\otimes d}$, so each $\sG_j$ is a disjoint union of $c_j$ complete bipartite graphs. Then 
    \[
        r + \sum_{j=1}^d c_j \leq 2^{d-1} + d
    \]
    with equality only if $r = 2^{d-1}$.
\end{Conjecture}

\Cref{tab:combinatorial_description_222} shows that \Cref{conj:maximal_dimension_binary} is true for $d=3$; we compute the graphical descriptions of all two-orthogonal decompositions for $2\times 2 \times 2$ tensors. Using an analogous approach, one can check that this conjecture is also true for $d=4$. In what follows, we provide further evidence for \Cref{conj:maximal_dimension_binary} and study its consequences.

\begin{table}[ht]
    \centering
    \newcolumntype{C}{>{\centering\arraybackslash} m{2cm} }
    \begin{tabular}{|C|C|C|C|}
        \hline
        & $\sG_1$ & $\sG_2$ & $\sG_3$ \\ \hline
        $W_1$ & 
        \begin{tikzpicture}
            \node[draw=none] (n0) at (0, 0.3) {};
            \node[draw,circle,minimum size=0.5cm,inner sep=0] (n1) at (0,0) {1};
        \end{tikzpicture} & 
        \begin{tikzpicture}
            \node[draw=none] (n0) at (0, 0.3) {};
            \node[draw,circle,minimum size=0.5cm,inner sep=0] (n1) at (0,0) {1};
        \end{tikzpicture} & 
        \begin{tikzpicture}
            \node[draw=none] (n0) at (0, 0.3) {};
            \node[draw,circle,minimum size=0.5cm,inner sep=0] (n1) at (0,0) {1};
        \end{tikzpicture} \\ \hline

        $\mathrm{odeco}$ & 
        \begin{tikzpicture}
            \node[draw=none] (n0) at (0, 0.3) {}; 
            \node[draw,circle,minimum size=0.5cm,inner sep=0] (n1) at (0,0) {1};
            \node[draw,circle,minimum size=0.5cm,inner sep=0] (n2) at (1,0) {2};
            \draw (n1) -- (n2);
        \end{tikzpicture} & 
        \begin{tikzpicture}
            \node[draw=none] (n0) at (0, 0.3) {};
            \node[draw,circle,minimum size=0.5cm,inner sep=0] (n1) at (0,0) {1};
            \node[draw,circle,minimum size=0.5cm,inner sep=0] (n2) at (1,0) {2};
            \draw (n1) -- (n2);
        \end{tikzpicture} & 
        \begin{tikzpicture}
            \node[draw=none] (n0) at (0, 0.3) {};
            \node[draw,circle,minimum size=0.5cm,inner sep=0] (n1) at (0,0) {1};
            \node[draw,circle,minimum size=0.5cm,inner sep=0] (n2) at (1,0) {2};
            \draw (n1) -- (n2);
        \end{tikzpicture} \\ \hline
        
        $W_2^{\{2,3\}}$ & 
        \begin{tikzpicture}
            \node[draw=none] (n0) at (0, 0.3) {}; 
            \node[draw,circle,minimum size=0.5cm,inner sep=0] (n1) at (0,0) {1};
            \node[draw,circle,minimum size=0.5cm,inner sep=0] (n2) at (1,0) {2};
        \end{tikzpicture} & 
        \begin{tikzpicture}
            \node[draw=none] (n0) at (0, 0.3) {};
            \node[draw,circle,minimum size=0.5cm,inner sep=0] (n1) at (0,0) {1};
            \node[draw,circle,minimum size=0.5cm,inner sep=0] (n2) at (1,0) {2};
            \draw (n1) -- (n2);
        \end{tikzpicture} & 
        \begin{tikzpicture}
            \node[draw=none] (n0) at (0, 0.3) {};
            \node[draw,circle,minimum size=0.5cm,inner sep=0] (n1) at (0,0) {1};
            \node[draw,circle,minimum size=0.5cm,inner sep=0] (n2) at (1,0) {2};
            \draw (n1) -- (n2);
        \end{tikzpicture} \\ \hline
        
        $W_2^{\{1,3\}}$ & 
        \begin{tikzpicture}
            \node[draw=none] (n0) at (0, 0.3) {};
            \node[draw,circle,minimum size=0.5cm,inner sep=0] (n1) at (0,0) {1};
            \node[draw,circle,minimum size=0.5cm,inner sep=0] (n2) at (1,0) {2};
            \draw (n1) -- (n2);
        \end{tikzpicture} & 
        \begin{tikzpicture}
            \node[draw=none] (n0) at (0, 0.3) {};
            \node[draw,circle,minimum size=0.5cm,inner sep=0] (n1) at (0,0) {1};
            \node[draw,circle,minimum size=0.5cm,inner sep=0] (n2) at (1,0) {2};
        \end{tikzpicture} & 
        \begin{tikzpicture}
            \node[draw=none] (n0) at (0, 0.3) {};
            \node[draw,circle,minimum size=0.5cm,inner sep=0] (n1) at (0,0) {1};
            \node[draw,circle,minimum size=0.5cm,inner sep=0] (n2) at (1,0) {2};
            \draw (n1) -- (n2);
        \end{tikzpicture} \\ \hline
        
        $W_2^{\{1,2\}}$ & 
        \begin{tikzpicture}
            \node[draw=none] (n0) at (0, 0.3) {};
            \node[draw,circle,minimum size=0.5cm,inner sep=0] (n1) at (0,0) {1};
            \node[draw,circle,minimum size=0.5cm,inner sep=0] (n2) at (1,0) {2};
            \draw (n1) -- (n2);
        \end{tikzpicture} & 
        \begin{tikzpicture}
            \node[draw=none] (n0) at (0, 0.3) {};
            \node[draw,circle,minimum size=0.5cm,inner sep=0] (n1) at (0,0) {1};
            \node[draw,circle,minimum size=0.5cm,inner sep=0] (n2) at (1,0) {2};
            \draw (n1) -- (n2);
        \end{tikzpicture} & 
        \begin{tikzpicture}
            \node[draw=none] (n0) at (0, 0.3) {};
            \node[draw,circle,minimum size=0.5cm,inner sep=0] (n1) at (0,0) {1};
            \node[draw,circle,minimum size=0.5cm,inner sep=0] (n2) at (1,0) {2};
        \end{tikzpicture} \\ \hline
        
        $W_3 \setminus W_2$ & 
        \begin{tikzpicture}
            \node[draw=none] (n0) at (0, 1.16) {};
            \node[draw,circle,minimum size=0.5cm,inner sep=0] (n1) at (0,0) {1};
            \node[draw,circle,minimum size=0.5cm,inner sep=0] (n2) at (1,0) {2};
            \node[draw,circle,minimum size=0.5cm,inner sep=0] (n3) at (0.5,0.86) {3};
            \draw (n1) -- (n3);
            \draw (n1) -- (n2);
        \end{tikzpicture} & 
        \begin{tikzpicture}
            \node[draw=none] (n0) at (0, 1.16) {};
            \node[draw,circle,minimum size=0.5cm,inner sep=0] (n1) at (0,0) {1};
            \node[draw,circle,minimum size=0.5cm,inner sep=0] (n2) at (1,0) {2};
            \node[draw,circle,minimum size=0.5cm,inner sep=0] (n3) at (0.5,0.86) {3};
            \draw (n1) -- (n2);
            \draw (n2) -- (n3);
        \end{tikzpicture} & 
        \begin{tikzpicture}
            \node[draw=none] (n0) at (0, 1.16) {};
            \node[draw,circle,minimum size=0.5cm,inner sep=0] (n1) at (0,0) {1};
            \node[draw,circle,minimum size=0.5cm,inner sep=0] (n2) at (1,0) {2};
            \node[draw,circle,minimum size=0.5cm,inner sep=0] (n3) at (0.5,0.86) {3};
            \draw (n1) -- (n3);
            \draw (n2) -- (n3);
        \end{tikzpicture} \\ \hline
        
        $W_4 \setminus W_3$ & 
        \begin{tikzpicture}
            \node[draw=none] (n0) at (0, 1.3) {};
            \node[draw,circle,minimum size=0.5cm,inner sep=0] (n1) at (0,0) {1};
            \node[draw,circle,minimum size=0.5cm,inner sep=0] (n2) at (1,0) {2};
            \node[draw,circle,minimum size=0.5cm,inner sep=0] (n3) at (1,1) {3};
            \node[draw,circle,minimum size=0.5cm,inner sep=0] (n4) at (0,1) {4};
            \draw (n1) -- (n2);
            \draw (n1) -- (n3);
            \draw (n2) -- (n4);
            \draw (n3) -- (n4);
        \end{tikzpicture} 
        & 
        \begin{tikzpicture}
            \node[draw=none] (n0) at (0, 1.3) {};
            \node[draw,circle,minimum size=0.5cm,inner sep=0] (n1) at (0,0) {1};
            \node[draw,circle,minimum size=0.5cm,inner sep=0] (n2) at (1,0) {2};
            \node[draw,circle,minimum size=0.5cm,inner sep=0] (n3) at (1,1) {3};
            \node[draw,circle,minimum size=0.5cm,inner sep=0] (n4) at (0,1) {4};
            \draw (n1) -- (n2);
            \draw (n1) -- (n4);
            \draw (n3) -- (n2);
            \draw (n3) -- (n4);
        \end{tikzpicture} &
        \begin{tikzpicture}
            \node[draw=none] (n0) at (0, 1.3) {};
            \node[draw,circle,minimum size=0.5cm,inner sep=0] (n1) at (0,0) {1};
            \node[draw,circle,minimum size=0.5cm,inner sep=0] (n2) at (1,0) {2};
            \node[draw,circle,minimum size=0.5cm,inner sep=0] (n3) at (1,1) {3};
            \node[draw,circle,minimum size=0.5cm,inner sep=0] (n4) at (0,1) {4};
            \draw (n1) -- (n3);
            \draw (n1) -- (n4);
            \draw (n2) -- (n3);
            \draw (n2) -- (n4);
        \end{tikzpicture} \\ \hline
    \end{tabular}
    \vspace{1em}
    \caption{Graphical descriptions of the two-orthogonal variety in $(\mathbb{R}^2)^{\otimes 3}$. Vertices are summands in a two-orthogonal decomposition. An edge in $\sG_k$ represents an orthogonality in the $k$-th factor. The table shows that  $W_1 \! \subseteq \!\textrm{odeco} \! \subseteq \! W_2$ and $\overline{W_3\setminus W_2} \subset \overline{W_4\setminus W_3}$. See \Cref{sec:2x2x2} for more details. }
    \label{tab:combinatorial_description_222}
\end{table}

\begin{Proposition} \label{prop:covering_bipartite_2}
    The complete multigraph $2K_r$ can be expressed as the union of $d$ bipartite graphs if and only if $r \leq 2^{d-1}$.
\end{Proposition}

\begin{proof}
    This is equivalent to \Cref{thm:maximal_length} for binary tensors.
\end{proof}

\begin{Proposition} \label{prop:covering_maximal}
    Fix a tuple of graphs $(\sG_1, \dots, \sG_d)$ on vertex set $[2^{d-1}]$. Assume that it is a valid graphical description for $W \subset (\R^2)^{\otimes d}$. Then each $\sG_j$ is a complete bipartite graph $K_{2^{d-2}, 2^{d-2}}$. This graphical description satisfies the bound from \Cref{conj:maximal_dimension_binary} with equality.
\end{Proposition}

\begin{proof}
    First, we show that $\sG_{d}$ is connected, and the same argument works for any $\sG_j$. Since each $\sG_j$ is bipartite, they are two-colorable. We label each vertex with a binary string of length $d-1$ such that if two vertices are connected in $\sG_j$, their labels differ in the $j$-th position.
    All vertices have distinct labels, as follows. Two vertices with the same label are not connected in any of $\sG_1, \dots, \sG_{d-1}$. So even if they were connected in $\sG_{d}$, there is at most one edge between them. Hence each of the $2^{d-1}$ possible labels is used exactly once.
    Moreover, if two labels differ in only one position, then they are only connected in one of $\sG_1, \ldots, \sG_{d-1}$, hence they are connected in $\sG_d$. 

    We construct a path between any two vertices $u$ and $v$ in $\sG_d$, as follows.
     Let $u$ and $v$ be labelled by binary strings $\ell_u$ and $\ell_v$ in $\{ 0, 1 \}^{d-1}$. We can convert $\ell_u$ into $\ell_v$ via a sequence of binary strings, each differing from the previous one in only one position. Since every binary string is a vertex in our graph, this is a sequence of vertices in the graph. All vertices with labels that differ in only one position are connected in $\sG_d$. Hence it is a path in $\sG_d$. 

     We have shown that each $\sG_j$ is connected. Let $A,B$ be two independent sets of $\sG_j$, so that every edge in $\sG_j$ connects a vertex from $A$ with a vertex from $B$. Then $|A| = |B| = 2^{d-2}$, otherwise we would need to connect more than $2^{d-2}$ vertices, with at least two edges for every pair of vertices, using only $d-1$ graphs, a contradiction to \Cref{prop:covering_bipartite_2}.
\end{proof}

\begin{Corollary}\label{cor:maximal_binary}
    If the difference $W_{2^{d-1}} \setminus W_{2^{d-1}-1} \subset (\R^2)^{\otimes d}$ is nonempty, then it consists of basis-aligned tensors (up to closure), so $\dim\left( \overline{W_{2^{d-1}} \setminus W_{2^{d-1}-1}} \right) = 2^{d-1} + d$.
\end{Corollary}

\begin{proof}
    The first part of the statement follows from \Cref{prop:covering_maximal}. The second part follows from \Cref{thm:dimension_lower_bound} and \Cref{prop:expected-dim}.
\end{proof}

\begin{Corollary}
    The variety $\overline{W_{2^{d-1}} \setminus W_{2^{d-1}-1}} \subset (\R^2)^{\otimes d}$ is irreducible.
\end{Corollary}
\begin{proof}
    The statement is vacuously true if $W_{2^{d-1}} \setminus W_{2^{d-1}-1} = \varnothing$. 
    Otherwise, \Cref{cor:maximal_binary} implies that $\overline{W_{2^{d-1}} \setminus W_{2^{d-1}-1}}$ consists of basis-aligned tensors. In particular, a two-orthogonal decomposition with $2^{d-1}$ summands comes from a Latin hypercube. Two Latin cubes are isotopic if one can be obtained from the other by permuting its rows, columns, and symbols. Similarly, two Latin hypercubes are called isotopic if one can be obtained from the other by permuting its mode-$k$ fibers (e.g. rows, columns, tubes, etc.) and symbols. Two-orthogonal decompositions coming from two isotopic Latin squares lie in the same irreducible component since permuting mode-$k$ fibers and symbols can be seen as acting with an element of $\SO(n)$. The number of isotopy classes of a Latin $d$-cube on $n=2$ symbols is 1, by induction. The case $d=2$ (i.e. a Latin square) is well known. A 3-cube can be thought of as a 2-cube on symbols $(1,2)$ and $(2,1)$. The same reasoning extends to higher-dimensional cubes.
\end{proof}

The following two results provide more evidence for \Cref{conj:maximal_dimension_binary}.

\begin{Proposition} \label{prop:weak-upper-bound}
    Fix a tuple of graphs $(\sG_1, \dots, \sG_d)$ on vertex set $[r]$. Assume that it is a valid graphical description for $W \subset (\R^2)^{\otimes d}$. Let $\sG_d$ be a disjoint union of $c_d$ complete bipartite graphs. Then $c_d + r \leq 2^{d-1} + 1$, with equality only if $r=2^{d-1}$.
\end{Proposition}

\begin{proof}
If $r = 2^{d-1}$ then $c_d=1$, by \Cref{prop:covering_maximal}. Suppose that $r < 2^{d-1}$. Each vertex in~$\sG_d$ can be labeled with a binary string of length $d-1$, given by two-colorings of $\sG_1, \dots, \sG_{d-1}$. All labels are distinct and labels at Hamming distance one are connected in $\sG_d$. We show that for every connected component in $\sG_d$ there is a binary string that cannot be a label for a vertex in $\sG_d$.
Consider a Hamiltonian cycle on the hypercube $\{0,1\}^{d-1}$, where vertices of Hamming distance one are connected. This induces a linear ordering $\ell_{1} \prec \cdots \prec \ell_{2^{d-1}}$. Since $r <  2^{d-1}$ we can assume that $\ell_1$ does not label any vertex in $\sG_d$. Given a connected component $\sH_i$ of $\sG_{d}$, let $\ell_{m_i}$ be the minimal label appearing in $\sH_i$. Then $\ell_{m_i - 1}$ cannot label any vertex in $\sG_{d}$. Of course, for different connected components, the corresponding minimal labels are different. Therefore, $c_d + r \leq 2^{d-1} < 2^{d-1} + 1$.
\end{proof}

\begin{Corollary}
    Fix a tuple of graphs $(\sG_1, \dots, \sG_d)$ on vertex set $[r]$. Assume that it is a valid graphical description for $W \subset (\R^2)^{\otimes d}$. Suppose that $\sG_1, \cdots, \sG_{d-1}$ are connected. If $\sG_d$ is not connected, then 
    \[\dim{\overline{\mathrm{Im}(\psi_{(\sG_1, \dots, \sG_d)})}} < 2^{d-1} + d.
    \]
\end{Corollary}

\begin{proof}
    It follows from \Cref{prop:expected-dim} and \Cref{prop:weak-upper-bound}.
\end{proof}

The previous results combine to give the following.

\begin{Corollary}
    If \Cref{conj:maximal_dimension_binary} is true, then $W_{2^{d-1} - 1} \neq W_{2^{d-1}}$ and $\overline{W_{2^{d-1}} \setminus W_{2^{d-1} - 1}}$ is an irreducible variety. Moreover, $\dim\overline{W_{2^{d-1}} \setminus W_{2^{d-1} - 1}}= 2^{d-1} + d$ and $\dim W_{2^{d-1} - 1} < 2^{d-1} + d$. 
\end{Corollary}

\section{$2 \times 2 \times 2$ tensors}\label{sec:2x2x2}

\subsection{Stratification and equations of the two-orthogonal variety.} In this section we study the two-orthogonal variety $W \subseteq (\R^{2})^{\otimes 3}$. We describe the varieties in the chain
\[
W_1 \subseteq W_2 \subseteq W_3 \subseteq W_4 = W.
\]
In \Cref{tab:combinatorial_description_222} we saw a combinatorial description of the two-orthogonal variety in $(\R^2)^{\otimes 3}$. Here we describe this variety algebraically. Consider the Euclidean inner product in $\R^2$ and let~$\{e_0, e_1\}$ be an orthonormal basis of $\R^2$. Using these coordinates, we express a tensor as~$\T = \sum_{i,j,k\in\{0,1\}} t_{ijk} e_i \otimes e_j \otimes e_k$. The equations shown below were obtained with \texttt{Macaulay2}.

Recall that each of the $W_r$'s is invariant under the action of $G = \SO(2)^{\times 3}$ on $(\R^{2})^{\otimes 3}$. This gives a concise way to describe our varieties: we only give a representative of each orbit. Let $u \in \R^2$ be an arbitrary unit vector, and let~$\lambda_1, \lambda_2, \lambda_3, \lambda_4 \in \R$ be arbitrary scalars. 

First, we have $W_1 = X$, the cone over the Segre variety. The tensors in $W_1$ are of the form
\[
\lambda_1 e_0 \otimes e_0 \otimes e_0 \in W_1,
\]
up to the action of $G$. The equations defining $W_1$ are the $2\times 2$ minors of all the flattenings \cite[Example 2.11]{JH}. 

Tensors in $W_2 \setminus W_1$ consist of two summands that are orthogonal in at least two factors. Let us denote by $W_2^{\{i,j\}}$ the tensors in $W_2$ whose two-orthogonal decompositions are orthogonal in the $i$-th and $j$-th factors. We have three families of tensors:
\begin{align*}
    \lambda_1 e_0 \otimes e_0 \otimes e_0 + \lambda_2 u \otimes e_1 \otimes e_1 & \in W_2^{\{2,3\}},\\
    \lambda_1 e_0 \otimes e_0 \otimes e_0 + \lambda_2 e_1 \otimes u \otimes e_1 & \in W_2^{\{1,3\}},\\
    \lambda_1 e_0 \otimes e_0 \otimes e_0 + \lambda_2 e_1 \otimes e_1 \otimes u & \in W_2^{\{1,2\}}.
\end{align*}
The equations defining each of these irreducible components are 
    \[
    W_2^{\{ 2,3\}} : \quad 
    \begin{cases}
    -t_{000}t_{101} + t_{100}t_{001} - t_{010}t_{111} + t_{110}t_{011} = 0\\
    
    -t_{000}t_{110} + t_{100}t_{010} - t_{001}t_{111} + t_{101}t_{011} = 0,
    \end{cases}
    \] 
and we obtain the equations for the other two by swapping indices. Each component is defined by a subset of the equations defining the odeco tensors, as discussed in \Cref{rk:equations}. The intersection of the three components is the odeco variety.

Given $\T = \lambda_1 e_0 \otimes e_0 \otimes e_0 + \lambda_2 u \otimes e_1 \otimes e_1 \in W_2$, to add a third summand to $\T$ satisfying two-orthogonality we need $u = \pm e_0$. We can then add $\lambda_3 e_1 \otimes e_0 \otimes e_1$ or~$\lambda_3 e_1 \otimes e_1 \otimes e_0$. Both two-orthogonal decompositions have the same orthogonality pattern between summands, up to permutation. If we repeat this analysis for the other two components of $W_2$ we also get equivalent tensors. Hence $\overline{W_3 \setminus W_2}$ is irreducible and consists of tensors of the form
\[
    \lambda_1 e_0 \otimes e_0 \otimes e_0 + \lambda_2 e_0 \otimes e_1 \otimes e_1 + \lambda_3 e_1 \otimes e_0 \otimes e_1.
\]
The variety $\overline{W_3 \setminus W_2}$ is defined by two polynomials. One is Cayley's hyperdeterminant:
\begin{align*}\label{eq:poly_Det}
    \Det(\T) =&
    \quad
    t_{000}^{2}t_{111}^{2} +  t_{001}^{2}t_{110}^{2} + t_{010}^{2}t_{101}^{2} + t_{100}^{2} t_{011}^{2}  \\
    &  -2\,t_{000}t_{001}t_{110}t_{111}
    -2\,t_{000}t_{010}t_{101}t_{111}
    -2\,t_{000}t_{100}t_{011}t_{111}\\
    &  -2\,t_{001}t_{100}t_{011}t_{110} -2\,t_{001}t_{010}t_{101}t_{110}
    -2\,t_{010}t_{100}t_{011}t_{101}\\
    & +4\,t_{000}t_{011}t_{101}t_{110} + 4\,t_{001}t_{010}t_{100}t_{111}.
\end{align*}
The other polynomial is a quartic with 40 monomials. This quartic is invariant under the action of $\SO(2)^{\times 3}$. In particular, it is invariant under flipping indices $0 \leftrightarrow 1$. After such relabeling, there are 5 distinct monomials. Consider the homogenization of the polynomial that defines the elliptope \cite{LAURENT1995439}:
\begin{equation} \label{eq:elliptope}
    g(z_1,z_2,z_3,z_4) = 2\,z_1z_2z_3 + z_1^2z_4 + z_2^2z_4 + z_3^2z_4 - z_4^3.
\end{equation}
Then, the other quartic that defines $\overline{W_3 \setminus W_2}$ is
\begin{equation}\label{eq:polyw4}
    f(\T) = \sum_{(i,j,k) \in \{0,1\}^3} (-1)^{i+j+k} \, 
    t_{i,j,k} \, g(t_{i+1,j,k}, t_{i, j+1, k}, t_{i,j,k+1}, t_{i+1, j+1, k+1}),
\end{equation}
where the sum in the indices is taken modulo $2$. Hence, $\overline{W_3\setminus W_2} = Z(f)\cap Z(\Det)$.

Finally, the tensors in $W_4\setminus W_3$ are of the form
\[
\lambda_1 e_0 \otimes e_0 \otimes e_0 + \lambda_2 e_0 \otimes e_1 \otimes e_1 + \lambda_3 e_1 \otimes e_0 \otimes e_1 +
\lambda_4 e_1 \otimes e_1 \otimes e_0,
\]
and $\overline{W_4 \setminus W_3}$ is the zero locus of $f$, from \ref{eq:polyw4}. 
The above combines to give the following. 

\begin{Proposition} \label{prop:stratification-222}
\Cref{tab:geometric_description_222} describes the stratification of the two-orthogonal variety in $(\R^2)^{\otimes 3}$. We have the inclusions $W_1 \subset \mathrm{odeco} \subset W_2$ and $\overline{W_3\setminus W_2} \subset \overline{W_4\setminus W_3}$. Hence, $W = \overline{W_4 \setminus W_3} \cup W_2$. 
\end{Proposition}

\begin{table}[ht]
\centering
\begin{tabular}{l|cccc}
variety & codimension & degree & \# components & generators per component\\
\hline
$W_1 = X$ & 4 & 6 & 1 & 9 quadrics \\
$\mathrm{odeco}$   & 3           & 8      & 1             & 3 quadrics\\
$W_2 = \overline{W_2 \setminus W_1}$   & 2           & 12     & 3             & 2 quadrics\\
$\overline{W_3 \setminus W_2}$   & 2           & 16     & 1     & 2 quartics \\
$\overline{W_4 \setminus W_3}$   & 1           & 4     & 1   & 1 quartic \\
$W = \overline{W_4 \setminus W_3} \cup W_2$ & 1 & 4 & 4 & 1 quartic / 2 quadrics
\end{tabular}
\caption{Description of the two-orthogonal variety in $(\mathbb{R}^2)^{\otimes 3}$.}
\label{tab:geometric_description_222}
\end{table}

Given $\T \in (\R^2)^{\otimes 3}$, we seek the closest two-orthogonal tensor to $\T$. To do so, we compute the critical points of $\|\T - \sS \|^2$ over $\sS \in W$ and choose the one that minimizes this quantity. The \emph{Euclidean distance degree} (ED degree) of a variety is the number of critical points of the squared Euclidean distance to a general point outside the variety \cite{draisma2016euclidean}. For example, the ED degree of the variety of rank-one tensors $X \subset (\R^2)^{\otimes 3}$ is six \cite{FO14}. Applying the technique in \cite{draisma2016euclidean}, we compute $\mathrm{EDdegree}(\overline{W_4\setminus W_3}) = 12$ and $\mathrm{EDdegree}(W_2^{\{i,j\}}) = 4$ for all $i\neq j \in [3]$. Since the ED degree is additive over the irreducible components, the ED degree of the two-orthgonal variety $W = \overline{W_4\setminus W_3} \cup W_2 \subset (\R^2)^{\otimes 3}$ is $24$.

\subsection{Generic identifiability of two-orthogonal decompositions.}

In \Cref{ex:non-unique-decomposition} we saw that two-orthogonal decompositions are not always unique. We show that they are unique for sufficiently general $2 \times 2 \times 2$ tensors.

\begin{Remark} 
    There exist symmetric two-orthogonal tensors that do not have a symmetric two-orthogonal decomposition.
    We will see that the symmetric tensor 
    \[
    \T = \lambda e_0 \otimes e_0 \otimes e_0 + e_0 \otimes e_1 \otimes e_1 + e_1 \otimes e_0 \otimes e_1 + e_1 \otimes e_1 \otimes e_0 \in W_4 \cap S^3(\R^2) \subset (\R^2)^{\otimes 3}
    \]
    does not admit a two-orthogonal decomposition with fewer terms, for a sufficiently general~$\lambda$. 
    The real tensors with symmetric two-orthogonal decompositions are odeco tensors. Hence $\T$ does not admit a symmetric odeco decomposition.
\end{Remark}
    
\uniqueness*
\begin{proof}
The two-orthogonal variety in $(\R^{2})^{\otimes 3}$ has 4 irreducible components:
\[
W = \overline{W_4 \setminus W_3} \cup W_2^{\{2,3\}} \cup  W_2^{\{1,3\}} \cup  W_2^{\{1,2\}}.
\]
We show that a generic tensor from each component has six singular vector tuples and only those appearing in the two-orthogonal decomposition are orthogonal. This implies the uniqueness of the two-orthogonal decomposition. First, let $\T$ be a generic tensor from $\overline{W_4 \setminus W_3}$:
\[
\T = \lambda_1 e_0 \otimes e_0 \otimes e_0 + \lambda_2 e_0 \otimes e_1 \otimes e_1 + \lambda_3 e_1 \otimes e_0 \otimes e_1 +
\lambda_4 e_1 \otimes e_1 \otimes e_0.
\]
The singular vector tuples are $u \otimes v \otimes w$ such that
\[
\det \hspace{-0.1cm}\begin{pmatrix}
        \lambda_1v_0w_0 + \lambda_2v_1w_1 & u_0 \\
        \lambda_3v_0w_1 + \lambda_4v_1w_0 & u_1 
    \end{pmatrix} = \det \hspace{-0.1cm} \begin{pmatrix}
        \lambda_1u_0w_0 + \lambda_3u_1w_1 & v_0 \\ 
        \lambda_2u_0w_1 + \lambda_4u_1w_0 & v_1
    \end{pmatrix} = \det \hspace{-0.1cm} \begin{pmatrix}
        \lambda_1v_0u_0 + \lambda_4v_1u_1 & w_0 \\ 
        \lambda_2u_0v_1 + \lambda_3u_1v_0 & w_1
    \end{pmatrix} \hspace{-0.05cm} = \hspace{-0.05cm} 0.
\]
This gives linear conditions in the monomials $\{u_iv_jw_k\}$. Suppose that $u_0 = 0$ and $u_1 = 1$. If $v_0 = 0$ then $w_1 = 0$, which gives solution $e_1 \otimes e_1 \otimes e_0$. If $v_0 = 1$ then $w_0 = v_1 = 0$, which gives solution $e_1 \otimes e_0 \otimes e_1$. The case $u_1 = 0$ is analogous and we obtain the other two terms in our decomposition.

Now, fix $u_1 = v_1 = w_1 = 1$ and let $g$ be the homogenization of the elliptope (\ref{eq:elliptope}). We get
\[
\begin{cases}
    \begin{array}{@{\hspace{0cm}}l@{\hspace{0cm}}c@{\hspace{0cm}}l@{\hspace{0cm}}c@{\hspace{0cm}}l}
    u_0 - \frac{2\,\lambda_{1}\lambda_{2}\lambda_{3}+\lambda_{1}^{2}\lambda_{4}+\lambda_{2}^{2}\lambda_{4}+\lambda_{3}^{2}\lambda_{4}-\lambda_{4}^{3}}{\lambda_{1}^{2}\lambda_{2}-\lambda_{2}^{3}+\lambda_{2}\lambda_{3}^{2}+2\,\lambda_{1}\lambda_{3}\lambda_{4}+\lambda_{2}\lambda_{4}^{2}}\, w_0 &=& u_0 - \frac{g(\lambda_1,\lambda_2,\lambda_3, \lambda_4)}{g(\lambda_1,\lambda_3,\lambda_4, \lambda_2)}w_0 &=& 0\\
    v_0 - \frac{2\,\lambda_{1}\lambda_{2}\lambda_{3}+\lambda_{1}^{2}\lambda_{4}+\lambda_{2}^{2}\lambda_{4}+\lambda_{3}^{2}\lambda_{4}-\lambda_{4}^{3}}{\lambda_{1}^{2}\lambda_{3}+\lambda_{2}^{2}\lambda_{3}-\lambda_{3}^{3}+2\,\lambda_{1}\lambda_{2}\lambda_{4}+\lambda_{3}\lambda_{4}^{2}} \, w_0 &=& v_0 - \frac{g(\lambda_1,\lambda_2,\lambda_3, \lambda_4)}{g(\lambda_1,\lambda_2,\lambda_4, \lambda_3)}w_0 &=&0\\
    w_0^2 + \frac{\left(\lambda_{1}^{2}\lambda_{3}+\lambda_{2}^{2}\lambda_{3}-\lambda_{3}^{3}+2\,\lambda_{1}\lambda_{2}\lambda_{4}+\lambda_{3}\lambda_{4}^{2}\right)\left(\lambda_{1}^{2}\lambda_{2}-\lambda_{2}^{3}+\lambda_{2}\lambda_{3}^{2}+2\,\lambda_{1}\lambda_{3}\lambda_{4}+\lambda_{2}\lambda_{4}^{2}\right)}{\left(2\,\lambda_{1}\lambda_{2}\lambda_{3}+\lambda_{1}^{2}\lambda_{4}+\lambda_{2}^{2}\lambda_{4}+\lambda_{3}^{2}\lambda_{4}-\lambda_{4}^{3}\right)\left(\lambda_{1}^{3}-\lambda_{1}\lambda_{2}^{2}-\lambda_{1}\lambda_{3}^{2}-2\,\lambda_{2}\lambda_{3}\lambda_{4}-\lambda_{1}\lambda_{4}^{2}\right)} &=& w_0^2 - \frac{g(\lambda_1,\lambda_2,\lambda_4, \lambda_3)g(\lambda_1,\lambda_3,\lambda_4, \lambda_2)}{g(\lambda_1,\lambda_2,\lambda_3, \lambda_4)g(\lambda_2,\lambda_3,\lambda_4, \lambda_1)} &=& 0.
    \end{array}
\end{cases}
\]
So we obtain two solutions 
\[
\begin{pmatrix}
    u_0 \\ v_0 \\ w_0 
\end{pmatrix}
= \pm \begin{pmatrix}  {\left( \frac{g(\lambda_1,\lambda_2,\lambda_3, \lambda_4)\,g(\lambda_1,\lambda_2,\lambda_4, \lambda_3)}{g(\lambda_1,\lambda_3,\lambda_4,\lambda_2) \, g(\lambda_2,\lambda_3,\lambda_4,\lambda_1)}\right) }^\frac12\\
     {\left( \frac{g(\lambda_1,\lambda_2,\lambda_3,\lambda_4)\,g(\lambda_1,\lambda_3,\lambda_4,\lambda_2)}{g(\lambda_1,\lambda_2,\lambda_4,\lambda_3)\,g(\lambda_2,\lambda_3,\lambda_4,\lambda_1)}\right) }^\frac12 \\
     {\left( \frac{g(\lambda_1,\lambda_2,\lambda_4,\lambda_3)\,g(\lambda_1,\lambda_3,\lambda_4,\lambda_2)}{g(\lambda_1,\lambda_2,\lambda_3,\lambda_4)\,g(\lambda_2,\lambda_3,\lambda_4,\lambda_1)}\right) }^\frac12
    \end{pmatrix}.
\]
This gives six singular vector tuples, the number of singular vector tuples of a generic $2 \times 2 \times 2$ tensor \cite{FO14}.
For generic $\lambda_i$'s we have $u_0^2, v_0^2, w_0^2 \neq 0,1$, which implies that the two new critical two rank-one approximations are not orthogonal to each other or to the summands appearing in the decomposition.

Now let $\T$ be a generic tensor from $\W_2^{\{1,2\}}$. The argument works analogously for the other two components of $W_2$. After an orthogonal change of basis and rescaling,
\[
\T = e_0 \otimes e_0 \otimes e_0 + \lambda e_1 \otimes e_1 \otimes y
\]
with $y = (y_0, 1) \in \R^2$ and $\lambda \in \R$. The singular vector tuples are $u \otimes v \otimes w$ such that
\[
\det \hspace{-0.1cm} \begin{pmatrix}
        v_0w_0 & u_0 \\
        \lambda v_1 \langle y, w \rangle & u_1 
    \end{pmatrix} = \det \hspace{-0.1cm} \begin{pmatrix}
        u_0w_0 & v_0 \\ 
        \lambda u_1\langle y, w \rangle & v_1
    \end{pmatrix} = \det \hspace{-0.1cm} \begin{pmatrix}
        v_0u_0 + \lambda y_0 v_1u_1 & w_0 \\ 
        \lambda u_1v_1 & w_1
    \end{pmatrix} = 0.
\]
Suppose that $u_0 = 1$ and $u_1 = 0$. This, implies $v_1 = w_1 = 0$, which gives the solution $e_0 \otimes e_0 \otimes e_0$. The cases $v_1 = 0$ and $w_1 = 0$ are analogous. Now, fix $u_1 = v_1 = w_1 = 1$. We get five solutions. One of them is the second summand from the decomposition given before. The other four are
\[
\begin{pmatrix}
    u_0 \\
    v_0 \\
    w_0
\end{pmatrix} = \begin{pmatrix}
    \pm \left(\frac{\lambda^{2}y_{0}^{2}+\lambda^{2}-\lambda\,y_{0}}{1-\lambda\,y_{0}} \right)^{\frac{1}{2}}\\
    \pm \left(\frac{\lambda^{2}y_{0}^{2}+\lambda^{2}-\lambda\,y_{0}}{1-\lambda\,y_{0}} \right)^{\frac{1}{2}}\\
    \frac{\lambda}{1-\lambda\,y_{0}}
\end{pmatrix} , \quad 
\begin{pmatrix}
    u_0 \\
    v_0 \\
    w_0
\end{pmatrix} =
\begin{pmatrix}
    \mp \left(\frac{\lambda^{2}y_{0}^{2}+\lambda^{2}+\lambda\,y_{0}}{1+\lambda\,y_{0}} \right)^{\frac{1}{2}} \\
    \pm \left(\frac{\lambda^{2}y_{0}^{2}+\lambda^{2}+\lambda\,y_{0}}{1+\lambda\,y_{0}} \right)^{\frac{1}{2}} \\ 
    -\frac{\lambda}{1+\lambda\,y_{0}}
\end{pmatrix}.
\]
For generic $y_0, \lambda \in \R$, the only singular vector tuples orthogonal to each other are the ones from the decomposition given before. So the two-orthogonal decomposition is unique.
\end{proof}

Using an analogous approach, one can verify that a two-orthogonal decomposition with 8 summands in $(\R^2)^{\otimes 4}$ is unique. It would be interesting to generalize this result to bigger tensors. We present the following conjecture.

\begin{Conjecture}\label{thm:exact-stabilization}
    A generic two-orthogonal tensor in $\mathbb{R}^{n_1} \otimes \cdots \otimes \mathbb{R}^{n_d}$ has a unique two-orthogonal decomposition (up to reordering of the summands). In particular, the chain 
    \[
    W_1 \subseteq W_2 \subseteq \cdots \subseteq W_r \subseteq \cdots
    \]
    stabilizes exactly at $N = \min_{k \in [d]} \prod_{j \neq k} n_j$. That is, $W_{N-1} \neq W_N = W_{N+1} = \cdots$.
\end{Conjecture}

\subsection*{Open questions} In this work, we study tensors that can be decomposed via a sequence of critical rank-one approximations. 
We focus on the decompositions that are order-independent, which we show to be characterized by the two-orthogonal property. In addition to the conjecture above, there are several open directions for future investigation. We show in \Cref{prop:stratification-222} that a generic $2 \times 2 \times 2$ tensor does not have a two-orthogonal decomposition. We believe this is also true for larger higher-order tensors, since it is a closed property on the singular vector tuples, but do not know the dimension of the two-orthogonal variety and if it can fill the space of tensors. We propose to study this problem by introducing the notion of valid graphical descriptions, which also leaves open directions of future work. For binary tensors of order $d$, we conjecture that the dimension of the two-orthogonal variety is $2^{d-1} + d$ (\Cref{conj:maximal_dimension_binary}) and connect this to a question about bipartite graphs. Finally, we conjecture that the notions of rank and two-orthogonal rank coincide generically up to the generic rank (\Cref{conj:rank}) and we provide evidence for small tensors.

\section*{Acknowledgments}

We thank Jan Draisma and \`Alex Rodr\'iguez Garc\'ia for helpful discussions. \'A. Ribot was supported by a fellowship from ”la
Caixa” Foundation (ID 100010434), with fellowship code LCF/BQ/EU23/12010097, and by an RCCHU fellowship. E. Horobe\c{t} was supported by the Project “Singularities and Applications” - CF 132/31.07.2023 funded by the European Union - NextGenerationEU - through Romania’s National Recovery and Resilience Plan. A. Seigal was partially supported by the NSF (DMR-2011754). E.T. Turatti was partially supported by the project Pure Mathematics in Norway, funded by Trond Mohn Foundation and Tromsø Research Foundation. 

\bibliographystyle{alpha}
\bibliography{refs}

\end{document}